\documentclass[12pt]{amsart}

\usepackage{general}
\usepackage{fullpage}

\title{Brauer and Etale Homotopy Obstructions to Rational Points on Open Covers}

\author{David A. Corwin and Tomer M. Schlank}
\begin{document}
\maketitle
\setlength{\baselineskip}{16pt}




\section*{Abstract}

In 2010, Poonen gave the first example of failure of the local-global principle that cannot be explained by Skorobogatov's \'etale Brauer-Manin obstruction.
Motivated by this example, we show that the Brauer-Manin obstruction detects non-existence of rational points on a sufficiently fine Zariski open cover of any variety over an imaginary quadratic or totally real field. We provide some evidence for why this is expected to happen more generally over any number field, some of which relates to the section conjecture in anabelian geometry. We then prove a result about the behavior of the \'etale Brauer obstruction in fibrations of varieties using the \'etale homotopy obstruction of Harpaz and the second author. We finally use that result and other techniques to further analyze Poonen's example in light of our general results.

\tableofcontents

\part{Introduction and Setup}

\section{Introduction}
Given a variety $X$ over a global field $k$, a major problem is to decide whether $X(k)=\emptyset$. By \cite{poo09}, it suffices to consider the case that $X$ is smooth, projective, and geometrically integral. As a first approximation one can consider the set $X(\Ab_k) \supset X(k)$, where $\Ab_k$ is the adele ring of $k$. It is a classical theorem of Minkowski and Hasse that if $X$ is a quadric, then $X(\Ab_k)\neq \emptyset \Rightarrow X(k)\neq \emptyset$. When a variety $X$ satisfies this implication, we say that it satisfies the Hasse principle, or local-global principle. In the 1940s, Lind and Reichardt (\cite{lin40}, \cite{rei42}) gave examples of genus 1 curves that do not satisfy the Hasse principle. More counterexamples to the Hasse principle were given throughout the years, until in 1971 Manin \cite{man70} described a general obstruction to the Hasse principle that explained many of the counterexamples to the Hasse principle that were known at the time. The obstruction (known as the Brauer-Manin obstruction) is defined by considering a certain set $X(\Ab_k)^{\Br}$, such that $X(k) \subset X(\Ab_k)^{\Br}  \subset X(\Ab_k)$. If $X$ is a counterexample to the Hasse principle, we say that it is accounted for or explained by the Brauer-Manin obstruction if $\emptyset = X(\Ab_k)^{\Br}  \subset X(\Ab_k) \neq \emptyset$.

In 1999, Skorobogatov (\cite{sk99}) defined a refinement of the Brauer-Manin obstruction known as the \'{e}tale Brauer-Manin obstruction and used it to produce an example of a variety $X$ such that $X(\Ab_k)^{\Br}  \neq \emptyset$ but $X(k) = \emptyset$. Namely, he described a set $X(\Ab_k)^{\etbr}$ for which  $X(k) \subset X(\Ab_k)^{\etbr}  \subset X(\Ab_k)^{\Br}  \subset X(\Ab_k)$ and found a variety $X$ such that $\emptyset = X(\Ab_k)^{\etbr}  \subset X(\Ab_k)^{\Br} \neq \emptyset$.

In his paper \cite{poo08}, Poonen constructed the first example of a variety $X$ such that $\emptyset = X(k)  \subset X(\Ab_k)^{\etbr} \neq \emptyset$. However, Poonen's method of showing that  $X(k)=\emptyset$ relies on the details of his specific construction and is not explained by a new finer obstruction. While it was hoped (\cite{pal10}) that an \'etale homotopy obstruction might solve the problem, it was shown (\cite{hs13}) that this provides nothing new.

Therefore, one wonders if Poonen's counterexample can be accounted for by an additional refinement of $X(\Ab_k)^{\etbr} $. Namely, can one give a general definition of a set
$$ X(k) \subset X(\Ab_k)^{\mathrm{new}} \subset X(\Ab_k)^{\etbr} $$
such that Poonen's variety $X$ satisfies $X(\Ab_k)^{\mathrm{new}}= \emptyset$.

In this paper, we provide such a refinement and prove that it is necessary and sufficient over many number fields. The obstruction is essentially given by applying the finite abelian descent obstruction to open covers of $X$ or decompositions of $X$ as a disjoint union of locally closed subvarieties. We shall make this statement more precise in Section \ref{sec:results}. As an example of our results, we can prove the following:

\begin{cor}[of Theorem \ref{thm:fabsomek}]
Let $k$ be a totally real or quadratic imaginary field. Let $X/k$ be a variety for which $X(k) = \emptyset$. Then there is a Zariski open cover $X = \bigcup_i U_i$ such that $U_i(\Ab_{k})^{\Br}=\emptyset$ for all $i$. 
\end{cor}

For a summary of all of the variants of this result, the reader may proceed to Section \ref{sec:results}.

The cases of Theorem 3.3 in which $k$ is $\Qb$ or quadratic imaginary follow immediately from \cite[Corollary 3.10]{LiuXu15}. Thus these cases, which amount to the material in Section \ref{sec:qb_quad_imag}, are not new. Nonetheless, we approach it slightly differently, both in that we use the \'etale homotopy obstruction to relate Brauer-Manin to finite abelian descent, and in that we discuss it in the general context of going beyond the Brauer-Manin obstruction.


Part \ref{part:embeddings} contains the main proofs of these results, one method unconditional and one method assuming the section conjecture in anabelian geometry.

In Part \ref{part:fibration}, we introduce homotopy sections and use them to analyze what happens to the \'etale Brauer-Manin obstruction in fibrations. 
In particular, we prove:
\begin{thm}[Theorem \ref{thm:vsa_geom_fib}]
Let $f \colon X \to S$ be a geometric fibration (c.f. Definition \ref{defn:geom_fib}), and suppose that $S(k)=S(\Ab_k^f)^{\etbr}$ and that for all $a \in S(k)$, we have $X_a(k)=X_a(\Ab_k^f)^{\etbr}$. Suppose furthermore that $X$ and $S$ satisfy the technical conditions 2-5 of Theorem \ref{thm:vsa_geom_fib}. Then
\[
X(k) = X(\Ab_k^f)^{\etbr}.
\]
\end{thm}

In Section \ref{newsec:sec_conj_fib}, we state the homotopy version of the section conjecture and show that it behaves well in fibrations. We then use this to give an alternative proof of Corollary \ref{cor:fcovallk} using Artin's good neighborhoods.


In Part \ref{part:examples}, we discuss the example of \cite{poo08}. Using Theorem \ref{thm:vsa_geom_fib}, we explain why Poonen's example required singular fibers:

\begin{thm}[Theorem \ref{thm:singular_fiber_needed}]
Let $f\colon X \to C$ be a smooth proper family of Ch\^{a}telet surfaces over an elliptic curve $C$ with $|C(k)|<\infty$. Suppose that for all $a \in C(k)$, we have $X_a(k)= \emptyset$. Suppose furthermore that the Tate-Shafarevich group of $C$ has trivial divisible subgroup, and that for every real place $v$ of $k$, every $a \in S(k)$, and every $b \in X_a(k_v)$, the map $\pi_1(X(k_v),b) \to \pi_1(S(k_v),a)$ is surjective (vacuous if $k$ is totally imaginary). Then $X(\Ab_k)^{\etbr} = \emptyset$.
\end{thm}

In Section \ref{sec:vsa_strat_open}, we find a Zariski open cover of the example of \cite{poo08} and prove that its pieces have empty \'etale-Brauer set. In fact, we introduce a few variants of this result, depending on different hypotheses, summarized at the beginning of Section \ref{sec:vsa_strat_open}. One of these uses Theorem \ref{thm:vsa_geom_fib}.

Motivated by one of these examples, we introduce the notion of quasi-torsors (Definition \ref{defn:quasi_torsor}), which are surjective morphisms that are generically torsors but with possible ramification, in Section \ref{sec:ramified_covers}. This leads us to propose a new approach to computing obstructions for open subvarieties, known as the ramified \'etale-Brauer obstruction (Definition \ref{defn:ram_et_br}).

A version of this article is Part II of the first author's PhD thesis (\cite{corwinthesis}).



\subsection{Notation and Conventions}\label{noteconv}

Whenever we speak of a field $k$, we implicitly fix a separable closure $k_s$ throughout, which produces a geometric point $\spec(k_s) \to \spec(k)$ of $k$. We write $X/k$ to mean that $X$ is a scheme over $\spec{k}$, and we denote by $X^s$ the base change of $X$ to $k_s$, as in \cite{PoonenBook}. We let $G_k$ denote $\Gal(k_s/k)$. If $X$ is a scheme, $\Oc(X)$ denotes $\Oc_X(X)$.

All cohomology is taken in the \'etale topology unless otherwise stated, and fundamental group always denotes an \'etale fundamental group. The same is true for higher homotopy groups, once they are defined in Definition \ref{defn:prof_htpy_groups}. If $k$ is a field, we write $H^*(k,-)$ for Galois cohomology of the field $k$.

For a global field $k$, we let $\Ab_k$ denote the ring of adeles of $k$. For a place $v$ of $k$, we let $k_v$ denote the completion of $k$ at $v$. When $v$ is finite, we let $\Ov$ denote the ring of integers in $k_v$, $\mf_v$ the maximal ideal in $\Oc_v$, $\pi_v$ a generator of $\mf_v$, $\Fv$ the residue field of $\Ov$, and $q_v$ the size of $\Fv$. If $S$ is a subset of the set of places of $k$, we let $\Ab_{k,S} = {\prod_{v \in S}}' (k_v,\Oc_v)$. This must not be confused with $\Ab_{k}^S = {\prod_{v \notin S}}' (k_v,\Oc_v).$ We let $\Oc_{k,S}$ denote the $S$-integers of $k$. In general, we replace $S$ by the letter ``f'' when $S$ is the set of finite places of $k$. We also use $\Ab_k^f$ to denote $\Ab_{k,S}$ when $S$ is the set of finite places, and we believe this does not lead to any confusion.


If $X$ is a variety over a local field $k_v$, we let $X(k_v)_{\bullet}$ denote the set of connected components of $X(k_v)$ under the $v$-adic topology, and we similarly set $X(\Ab_{k,S})_{\bullet} = {\prod_{v \in S}}' (X(k_v)_{\bullet},X(\Oc_v)_{\bullet})$. Note that this changes only the Archimedean places.

\subsection{Acknowledgements}

The first author would like to thank Alexei Skorobogatov for discussing an earlier version of the paper, Gereon Quick for responding to queries about \'etale homotopy, and Jakob Stix for providing a reference. He would like to thank his host at ENS, Olivier Wittenberg, for going through parts of the paper and help on some of the arguments, especially on a model-theoretic argument in the proof of Theorem \ref{thm:vsa_example}. He would finally like to thank his thesis advisor Bjorn Poonen for an enormous amount of help, including multiple careful readings of the paper and numerous discussions about the ideas.

The first author was supported by NSF grants DMS-1069236 and DMS-160194, Simons Foundation grant \#402472, and NSF RTG Grant \#1646385 during various parts of the writing of this paper.

The second author would like to thank Jean-Louis Colliot-Th\'{e}l\`{e}ne and Alexei Skorobogatov for many useful discussions.

Some of the work presented here is based on earlier work by the second author done while attending at the ``Diophantine equations'' trimester program in 2009 at the Hausdorff Institute in Bonn. The author would like to thank the staff of the institute for providing a pleasant atmosphere and excellent working conditions.

The second author would also like to thank Yonatan Harpaz for his useful comments on an earlier draft of this paper.

Both authors would like to thank the Hebrew University of Jerusalem for support for the second author to visit the first author.

\section{Obstructions to Rational Points}

\subsection{Generalized Obstructions}

Let $k$ be a global field.

\begin{defn}
Let $\omega$ be a subfunctor of the functor $X \mapsto X(\Ab_k)$ from $k$-varieties to sets. We write $X(\Ab_k)^{\omega}$ instead of $\omega(X)$. We say that $\omega$ is a \emph{generalized obstruction (to the local-global principle)} if $X(k) \subseteq X(\Ab_k)^{\omega}$ for every $k$-variety $X$.
\end{defn}

If $S$ is a subset of the places of $k$, we define $X(\Ab_{k,S})^\omega$ to be the projection of $X(\Ab_k)^\omega$ from $X(\Ab_k)$ to $X(\Ab_{k,S})$.

We note the trivial but useful facts:

\begin{lemma}\label{lemma:combined1}
Let $k$ be a global field and $X$ a $k$-variety.
\begin{enumerate}[(i)]
    \item\label{item:S} If $\omega$ and $\omega'$ are two generalized obstructions, and $X(\Ab_k)^{\omega} \subseteq X(\Ab_k)^{\omega'}$, then $X(\Ab_{k,S})^{\omega} \subseteq X(\Ab_{k,S})^{\omega'}$.
    \item\label{item:change_places} If $S,S'$ are two nonempty sets of places of $k$, then $X(\Ab_{k,S})^\omega = \emptyset$ if and only if $X(\Ab_{k,S'})^\omega = \emptyset$.
    \item The association $X \mapsto X(\Ab_{k,S})^{\omega}$ is a subfunctor of the covariant functor from $k$-varieties to sets corepresented by $\spec(\Ab_{k,S})$.

\end{enumerate}
\end{lemma}

\begin{proof}
These all follow by simple diagram chase.
\end{proof}

\subsubsection{Very Strong Approximation}

\begin{defn}\label{obstructiondata} If $k$ is a global field, S is a nonempty set of places of k, and $\omega$ is a generalized obstruction, then we call $(\omega,S,k)$ an \emph{obstruction datum}.\end{defn}

When $S$ is the set of finite places of $k$, we also write $(\omega,f,k)$ in this case.

We will sometimes leave out $S$ and write $(\omega,k)$, in which case \emph{we understand $S$ to be the set of all places of $k$}. We often write $(\omega,S)$ when $k$ is understood.

\begin{defn}\label{def_vsa}Let $(\omega,S,k)$ be an obstruction datum. We say that a variety $X/k$ \emph{is} or \emph{satisfies} \emph{very strong approximation} (VSA) for $(\omega,S)$ if $X(k)=X(\Ab_{k,S})^{\omega}$.\end{defn}


\begin{lemma}\label{lemma:combined2}
Let $X$ and $Y$ be $k$-varieties and $W \subseteq X$ a (locally closed) subvariety of $X$ over $k$.

\begin{enumerate}[(i)]
    \item\label{item:embeds_vsa} If $X$ is VSA for $(\omega,S,k)$, then so is $W$.
    \item\label{item:prod_vsa} If $X$ and $Y$ are VSA for $(\omega,S,k)$, then so is $X \times Y$.
    \end{enumerate}

\end{lemma}

\begin{proof}
For (\ref{item:embeds_vsa}), note that $\alpha \in W(\Ab_{k,S})^{\omega}$. Then $\alpha \in X(\Ab_{k,S})^{\omega}$ by functoriality. Since $X$ is VSA, $\alpha \in X(k)$; i.e., $\spec{\Ab_{k,S}} \xrightarrow{\alpha} W$ factors through a map $\spec{k} \to X$. If $v \in S$, then the $v$-component $\alpha_v$ of $\alpha$ is a $k_v$-point of $X$ coming from a $k$-point of $X$. But because $\alpha \in W(\Ab_{k,S})$, this implies that $\alpha_v \in W(k_v)$. But if it is a $k$-point as a point of $X$, then we also have $\alpha_v \in W(k)$.

(\ref{item:prod_vsa}) follows by a simple diagram chase.
\end{proof}

\begin{rem}It might seem desirable in our main theorems to replace the set of points at the real places by the set of their connected components, as is done in \cite{sto07}, instead of using $X(\Ab_{k}^f)$, as we do. We do this precisely because Lemma \ref{lemma:combined2}(\ref{item:embeds_vsa}) would not otherwise work, as the map from $W(\Ab_k)_\bullet$ to $X(\Ab_k)_\bullet$ need not be injective.\end{rem}

\subsection{Functor Obstructions}

Let $k$ be a global field. Following the formalism in \cite[\S 8.1]{PoonenBook}, let $F$ be a contravariant functor from schemes over $k$ to sets (a.k.a. a presheaf of sets on the category of $k$-schemes). For $X$ a $k$-scheme and $A \in F(X)$, we have the following commutative diagram
$$
\begin{CD}
X(k) @>>> X(\Ab_k)\\
@VVV @VVV\\
F(k) @>\loc >> \prod_v F(k_v),
\end{CD}
$$
where the vertical arrows denote pullback of $A$ from $X$ to $k$ or $k_v$. We define $X(\Ab_k)^A$ as the subset of $X(\Ab_k)$ whose image in the lower right object is in the image of $\loc$. We then define the obstruction set
$$
X(\Ab_k)^F = \bigcap_{A \in F(X)} X(\Ab_k)^A.$$


If $\Fc$ is a collection of such functors $F$, we define $$X(\Ab_{k})^{\Fc} = \bigcap_{F \in \Fc} X(\Ab_{k})^F.$$

\begin{lemma}\label{lemma:combined3} The functor $X \mapsto X(\Ab_k)^{\Fc}$ is a generalized obstruction. Furthermore, if $\Fc' \subseteq \Fc$, then $X(\Ab_{k,S})^{\Fc} \subseteq X(\Ab_{k,S})^{\Fc'}$ for all $X$.\end{lemma}
\begin{proof}
It is clear that $X(k) \subseteq X(\Ab_k)^{\Fc}$, so it suffices to verify functoriality, which is a simply diagram chase. The second part is immediate for $\Ab_k$, and follows for $\Ab_{k,S}$ by Lemma \ref{lemma:combined1}(\ref{item:S}).

\end{proof}


\subsection{Descent Obstructions}\label{sec:descentobstructions}

Let $G$ be a group scheme over $k$. We let $F_G$ denote the contravariant functor from $k$-schemes to sets for which $F_G(X)$ is the set of isomorphism classes of $G$-torsors for the fppf topology over $X$. By \cite[Theorem 6.5.10(i)]{PoonenBook}, we have $F_G(X)=H^1_{\fppf}(X,G)$ when $G$ is affine. If $G$ is also smooth, it is isomorphic to the \'etale cohomology, denoted $H^1(X,G)$.


Let $\Gc$ be a subset of the set of isomorphism classes of finite type group schemes over $k$. For $\Fc = \{F_G\}_{[G] \in \Gc}$, we set $$X(\Ab_{k})^\Gc = X(\Ab_{k})^\Fc.$$

Usually, $\Gc$ will consist only of algebraic groups over $k$, i.e., finite type separated group schemes over $k$.

\begin{defn} For $\Gc$ the set of smooth affine algebraic groups over $k$, we call this the \emph{descent obstruction} and denote it by $X(\Ab_{k})^{\mathrm{descent}}$ (following \cite{PoonenBook}).\end{defn}

\begin{defn}\label{defn:fcovsolab} For $\Gc$ the set of (all, solvable, abelian) finite smooth algebraic groups over $k$, we call this the (\emph{finite descent}, \emph{finite solvable descent}, \emph{finite abelian descent}) \emph{obstruction} and denote it by ($X(\Ab_{k})^{\fcov}$, $X(\Ab_{k})^{\fsol}$, $X(\Ab_{k})^{\fab}$), respectively (following notation in \cite{sto07},\cite{sk09}).\end{defn}




\begin{rem} By the equivalence between (i) and (i') in \cite[Theorem 2.1]{harsx12}, it suffices to take only finite \emph{constant} algebraic groups over $k$ in the definition of $X(\Ab_{k})^{\fcov}$.\end{rem}



\begin{prop}\label{compare_obstructions} We have $X(\Ab_{k})^{\mathrm{descent}} \subseteq X(\Ab_{k})^{\fcov} \subseteq X(\Ab_{k})^{\fsol} \subseteq X(\Ab_{k})^{\fab}.$
\end{prop}
\begin{proof}
This follows by Lemma \ref{lemma:combined3} because the corresponding collections of isomorphism classes of group schemes become smaller as we progress from $\desc$ to $\fcov$ to $\fsol$ to $\fab$.
\end{proof}

\begin{cor}\label{compare_vsa}
For fixed $k$ and $S$, VSA for $\fab$ implies VSA for $\fsol$, which implies VSA for $\fcov$, which implies VSA for $\desc$.
\end{cor}
\begin{proof}
This follows immediately from Proposition \ref{compare_obstructions} and Lemma \ref{lemma:combined1}(\ref{item:S}).
\end{proof}





\subsubsection{Alternative Description in Terms of Images of Adelic Points}\label{sec:altdescr}

Let $G$ be a smooth affine algebraic group over $k$, and let $Z \in F_G(X)$ be given by, with abuse of notation, $f:Z \to X$. If $\tau \in F_G(k)$, one can define the twist $f^{\tau}\colon Z^\tau \to X$ as in \cite[Example 6.5.12]{PoonenBook}.

By \cite[Theorem 8.4.1]{PoonenBook}, if $x \in X(k)$, then $x \in f^{\tau} (Y^\tau(k))$ if and only if $\tau$ is the pullback of $Z$ under $x \colon \spec{k} \to X$.

It follows that $$X(k) = \bigcup_{\tau \in F_G(k)} f^{\tau} (Z^\tau(k)) \subseteq \bigcup_{\tau \in F_G(k)} f^{\tau} (Z^\tau(\Ab_{k,S})).$$

\begin{prop}
We have
$$\bigcup_{\tau \in F_G(k)} f^\tau(Z^\tau(\Ab_{k})) = X(\Ab_{k})^Z.$$
\end{prop}

\begin{proof}
Suppose $\alpha \in X(\Ab_{k})^Z$. We think of this point as a map $\alpha \colon \spec{\Ab_{k}} \to X$, and we know by definition that the pullback of $Z$ along this map is in the image of $$H^1(k,G) \xrightarrow{\loc} \prod_{v} H^1(k_v,G),$$ say of $\tau \in H^1(k,G)$. One can easily check that twisting commutes with pullback (from $k$ to $\Ab_{k}$). Thus the proof of \cite[8.4.1]{PoonenBook} tells us that the pullback of $Z^{\tau}$ under this adelic point is now the trivial element of $\prod_{v} H^1(k_v,G)$. But this means that the fiber of $Z^{\tau}$ over $\alpha$ is a trivial torsor and thus contains an adelic point; i.e., $\alpha \in f^\tau(Z^\tau(\Ab_{k}))$.

Conversely, suppose $\alpha \in f^\tau(Z^{\tau}(\Ab_{k}))$ for some $\tau \in H^1(k,G)$. Then for each $v$, we have $\alpha_v \in f^{\tau}(Z^{\tau}(k_v))$, so Theorem 8.4.1 tells us that the pullback of $Z$ under $\alpha_v \colon \spec{k_v} \to X$ is $\loc_v(\tau)$. But this implies that the pullback of $Z$ under $\alpha$ is in the image of $\loc$; i.e., $\alpha \in X(\Ab_{k})^Z$.
\end{proof}

\subsection{Comparison with Brauer and Homotopy Obstructions}\label{otherobstructions}

There is the Brauer-Manin obstruction set, $X(\Ab_{k,S})^{\Br}$, which is the obstruction for $F(-)=H^2(-,\Gb_m)$.

There is also the \'etale Brauer-Manin obstruction $$X(\Ab_{k})^{\etbr} = \bigcap_{\substack{\mathrm{finite}\,\acute{ \mathrm{e}}\mathrm{tale}\,G\\ \mathrm{all} \, G\mathrm{-torsors} f:Y \to X}}
\bigcup_{\tau \in H^1(k,G)}  f^\tau(Y^\tau(\Ab_{k})^{\Br}).$$

By Section \ref{sec:altdescr}, this is contained in $X(\Ab_{k,S})^{\fcov}$.

A series of obstructions is defined in \cite{hs13} when $k$ is a number field:
$$
X(\Ab_k)^h \subseteq \cdots \subseteq X(\Ab_k)^{h,2} \subseteq X(\Ab_k)^{h,1} \subseteq X(\Ab_k)
$$
and
$$
X(\Ab_k)^{\Zb h} \subseteq \cdots \subseteq X(\Ab_k)^{\Zb h,2} \subseteq X(\Ab_k)^{\Zb h,1} \subseteq X(\Ab_k)
$$
such that $X(\Ab_k)^{h,n} \subseteq X(\Ab_k)^{\Zb h,n}$ for all $n$, $X(\Ab_k)^{\Zb h} = \bigcap_{n} X(\Ab_k)^{\Zb h,n}$, and $X(\Ab_k)^h = \bigcap_n X(\Ab_k)^{h,n}$. The definition of $X(\Ab_k)^h$ will be discussed in more detail in Section \ref{sec:hobstructions}. When discussing obstruction data, we write $(h,S,k)$ and $(\Zb h,S,k)$.

By \cite[Theorem 9.136]{hs13} and Lemma \ref{lemma:combined1}(\ref{item:S}), we have for a smooth geometrically connected variety $X$ that:
\begin{eqnarray*}
X(\Ab_{k,S})^h &=& X(\Ab_{k,S})^{\mathrm{\acute{e}t},\Br}\\
X(\Ab_{k,S})^{\Zb h} &=& X(\Ab_{k,S})^{\Br}\\
X(\Ab_{k,S})^{h,1} &=& X(\Ab_{k,S})^{\fcov}\\
X(\Ab_{k,S})^{\Zb h,1} &=& X(\Ab_{k,S})^{\mathrm{f-ab}}
\end{eqnarray*}

In particular, we have the two inclusions: $$X(\Ab_{k,S})^{\mathrm{\acute{e}t},\Br} \subseteq X(\Ab_{k,S})^{\Br} \subseteq X(\Ab_{k,S})^{\fab}$$
$$X(\Ab_{k,S})^{\mathrm{\acute{e}t},\Br} \subseteq X(\Ab_{k,S})^{\fcov} \subseteq X(\Ab_{k,S})^{\fab}.$$

\section{Statements of Abstract Results and Conjectures}\label{sec:results}

For the rest of this section, we assume that $k$ is a number field.

The main reason to care about these obstructions is that they help prove that a variety $X/k$ has no rational points. The most powerful obstruction currently known is $X(\Ab_k)^{\etbr}=X(\Ab_k)^{\mathrm{descent}}$. But, as stated in the introduction, even in this case, there is a variety $X$ with $\emptyset = X(k) \subseteq X(\Ab_k)^{\etbr} \neq \emptyset$, as found in \cite{poo08}. In the method of proof that $X(k)=\emptyset$, it is clear that the \'etale Brauer-Manin obstruction and its avatars still appear, but they are applied separately to different \emph{pieces} of $X$. From this point of view, it's natural to ask the following question:

\begin{quest}\label{quest:emptiness}
Let $X$ be a $k$-variety with $X(k)=\emptyset$. Does there exist a finite open cover or stratification of $X$ for which each constituent part has empty \'etale-Brauer set? More strongly, is the same true for any of the other obstruction sets from Section \ref{sec:descentobstructions}?
\end{quest}

If true, this proves that $X(k)=\emptyset$, because if each constituent part has no rational points, then $X$ does not.

One could ask for the stronger statement that each constituent part satisfies VSA:

\begin{quest}\label{quest:vsa}
Let $X$ be a $k$-variety and $(\omega,S,k)$ an obstruction datum. Does there exist a finite open cover or stratification of $X$ for which each constituent part satisfies VSA for $(\omega,S,k)$?
\end{quest}

In this paper, we obtain the following result.

\begin{thm}[Corollary \ref{cor:fabsomek} and Corollary \ref{cor:uisvsa}]\label{thm:fabsomek}
The answer to Question \ref{quest:vsa} is yes for $(\fab,f,k)$ when $k$ is a imaginary quadratic or totally real number field.
\end{thm}

As mentioned in the introduction, this is essentially \cite[Corollary 3.10]{LiuXu15} in the case $k$ is $\Qb$ or quadratic imaginary. Note that this implies the same result for all other obstructions defined so far (such as $\fsol$, $\fcov$, $\Br$, and $\mathrm{\acute{e}t-Br}$).


We conjecture this to hold for all number fields:

\begin{conj}[Conjecture \ref{conj:faballk}]
For any number field $k$, the answer to Question \ref{quest:vsa} is yes for $(\fab,f,k)$ 
\end{conj}

As progress toward this conjecture, we prove the following conditional result:

\begin{thm}[Corollary \ref{cor:fcovallk}]\label{thm:fcovallk}
Assuming Grothendieck's section conjecture (\ref{c:gro2}), the answer to Question \ref{quest:vsa} is yes for $(\fcov,S,k)$, where $S$ is a nonempty set of finite places.
\end{thm}


The primary method of proof consists in reducing the problem to proving VSA for an open subset of $\Pb^1$, a condition on $\omega$ and $k$ we call $A(\omega,f,k)$. Theorem \ref{thm:vsacover} then tells us in general that this answers Question \ref{quest:vsa} (and therefore Question \ref{quest:emptiness}) affirmatively. This is described in Part \ref{part:embeddings}.

\begin{rem}
As described in the introduction, our results were motivated by the search for a new obstruction. Here, we explain how the aforementioned results may be rephrased as results about a new obstruction.

Let $\Ccr$ be a collection of finite collections of locally closed subvarieties (such as the collection of all finite open covers of $X$, $\mathcal{OPEN}$, or the collection of all finite stratifications{We are defining a finite stratification of $X$ as a finite partially-ordered index set $I$ along with a locally closed subset $S_i \subseteq X$ for every $i \in I$ such that $X$ is the disjoint union of all $S_i$, and the closure of any given $S_i$ in $X$ is the union $\bigcup_{j \le i} S_j$.}, $\mathcal{STRAT}$). We may then set
$$
X(\Ab_{k,S})^{\Ccr, \omega} = \bigcap_{\Xcr \in \Ccr} \bigcup_{X_i \in \Xcr} X_i(\Ab_{k,S})^{\omega}.
$$

We can then rephrase Theorem \ref{thm:vsacover} as saying that if $A(\omega,S,k)$ holds, then 
\[X(k) = X(\Ab_{k,S})^{\mathcal{OPEN}, \omega}.\]

This is discussed in more detail in \cite[Section 4.3.1]{corwinthesis}.
\end{rem}

\begin{rem}
It is interesting to consider the constructions in this subsection applied when $\omega$ is such that $X(\Ab_k)^{\omega} = X(\Ab_k)$ for all $X$. This is discussed in \cite[Appendix A]{corwinthesis}.
\end{rem}

\begin{rem}
This refinement of obstructions via open covers is related to the notion of cosheafification. This connection is discussed in \cite[Appendix B]{corwinthesis}.
\end{rem}


We now explain how Parts \ref{part:fibration} and \ref{part:examples} relate to the aforementioned results. In Part \ref{part:fibration}, we introduce homotopy sections and analyze what happens to the \'etale Brauer-Manin obstruction in fibrations. We derive a result about the behavior of the \'etale-Brauer obstruction in fibrations, Theorem \ref{thm:vsa_geom_fib}. Our analysis of the section conjecture in fibrations also allows us to give a different proof of Theorem \ref{thm:fcovallk}, in Corollary \ref{cor:fcovallk2}. While this might seem logically unnecessary if the theorem is already proven, the second proof provides a possibly different open cover than that of the first proof.

In Part \ref{part:examples}, we review the original example of \cite{poo08} in terms of our results. The main results of that section are summarized in the introduction, and some of them are summarized at the beginning of Section \ref{sec:vsa_strat_open}, so we do not dwell on them here. As mentioned in the introduction, we also introduce the ramified \'etale-Brauer obstruction (Definition \ref{defn:ram_et_br}), which allows us to explicitly apply our stratified \'etale-Brauer obstruction to Poonen's example but in such a way that we compute the Brauer-Manin obstruction only for proper varieties.

\part{Main Result via Embeddings}\label{part:embeddings}

\section{General Setup}

We set up the basic formalism for the various unconditional and conditional results proven via embeddings.

\begin{dfthm}\label{dfthm:A}Let $(\omega,S,k)$ be an obstruction datum. The following statements are equivalent:
\begin{itemize}
\item[(i)] There is a nonempty open $k$-subscheme of $\Pb_k^1$ that is VSA for $(\omega,S,k)$.
\item[(ii)] There is a nonempty open $k$-subscheme of $\Ab_k^1$ that is VSA for $(\omega,S,k)$.
\item[(iii)] For each positive integer $n$, there is a nonempty open $k$-subscheme of $\Ab_k^n$ that is VSA for $(\omega,S,k)$.
\item[(iv)] For each positive integer $n$, there is a nonempty open $k$-subscheme of $\Pb_k^n$ that is VSA for $(\omega,S,k)$.

\end{itemize}

If this is the case, we say that the property $A(\omega,S,k)$ is true.

\end{dfthm}

\begin{proof}
If (i) holds, let this open subscheme be $U$. Then $V \colonequals U \cap \Ab^1_k$ is nonempty because any nonempty open is dense, and $V$ is then VSA by Lemma \ref{lemma:combined2}(\ref{item:embeds_vsa}), so (ii) holds. If (ii) holds, then $V^n$ is a nonempty open subscheme of $\Ab^n_k$ for all $n$, which is VSA by Lemma \ref{lemma:combined2}(\ref{item:embeds_vsa}), so (iii) holds. If (iii) holds, then the VSA nonempty open subscheme of $\Ab_k^n$ is also a VSA nonempty open subscheme of $\Pb_k^n$, so (iv) holds. Finally, (iv) implies (i) by setting $n=1$.
\end{proof}

We note that $\PGL_2(k)$ acts on $\Pb^1_k$ by $k$-automorphisms, and hence $(\PGL_2(k))^n$ acts on $(\Pb^1_k)^n$ in the same way.

\begin{lemma}\label{denseorbit}Let $k$ be an infinite field and $x$ a closed point of $(\Pb^1_k)^n$. Then the orbit of $x$ under $(\PGL_2(k))^n$ is Zariski dense in $(\Pb^1_k)^n$.\end{lemma}

\begin{proof}

We choose an algebraic closure $\cl{k}$ of $k$ and a point $\cl{x}$ of $(\Pb^1(\cl{k}))^n$ representing $x$. We equivalently wish to show that for any open $U \subseteq (\Pb^1_k)^n$, the orbit of $\cl{x}$ under $(\PGL_2(k))^n$ intersects $U(\cl{k})$.

We first prove this for $n=1$. The points $(\cl{x}+a)_{a \in k}$ of $\Pb^1$ are distinct. As there are infinitely many of them, and $\Pb^1$ is one-dimensional, they are Zariski dense. But these points are all in the orbit of $\cl{x}$, so this orbit is Zariski dense.

Let $\cl{x}=(x_1,\cdots,x_n)$ with $x_i \in \Pb^1_k(\cl{k})$, and let $S_i$ be the $\PGL_2(k)$-orbit of $x_i$. The previous paragraph implies that each $S_i$ is Zariski-dense in $\Pb^1(\cl{k})$, so $\prod_{i=1}^n S_i$ is dense in $(\Pb^1(\cl{k}))^n$.

A less detailed and less general version of this argument is given in \cite[Lemma 6.3]{schsx16}.

\end{proof}

\begin{thm}\label{thm:vsacover}
Suppose that $A(\omega,S,k)$ holds. Let $X$ be a $k$-variety. Then there exists a finite affine open cover $X = \bigcup_i V_i$ such that $V_i$ is VSA for $(\omega,S,k)$ for every $i$.
\end{thm}

\begin{proof}
Let $x$ be any closed point. By the definition of a scheme, there is an affine open neighborhood of $X$ containing $x$, so we can assume that $X$ is affine.

We now embed $X$ into $\Ab^n$ for some sufficiently large $n$. As $\Ab^1$ is an open subscheme of $\Pb^1$, we have an open inclusion $\Ab^n = (\Ab^1)^n \hookrightarrow (\Pb^1)^n$, so we get an embedding $\phi:X \hookrightarrow (\Pb^1)^n$.

Assuming $A(\omega,S,k)$, there is an open subset $U \subseteq (\Pb^1)^n$ that is VSA for $(\omega,S,k)$. By Lemma \ref{denseorbit}, there is a $k$-automorphism $g$ of $(\Pb^1)^n$ sending $\phi(x)$ into $U$. Then $U$ contains $g(\phi(x))$. But then $U \cap g(\phi(X))$ is an open subscheme of $X$ containing $x$ and is a locally closed subscheme of $U$. It is therefore VSA for $(\omega,S,k)$ by Lemma \ref{lemma:combined2}(\ref{item:embeds_vsa}). Choose an affine open neighborhood $V_x$ of $x$ in $U \cap g(\phi(X))$, and $V_x$ is again VSA.

We do this for every $x$, and we obtain an affine open cover $X = \bigcup_{x \in X} V_x$ such that $V_x$ is VSA for $(\omega,S,k)$ for every $x$. As varieties are quasi-compact, this has a finite subcover, which proves the theorem.
\end{proof}

\begin{cor}Let $X$ be a $k$-variety with $X(k)=\emptyset$, and suppose $A(\omega,S,k)$. Then there exists a finite open cover of $X$ such that each constituent open set has empty obstruction set for $(\omega,S,k)$. The same is true for any other nonempty set $S'$ of places of $k$.\end{cor}


\begin{proof}
Let $\bigcup_i V_i$ be an open cover as in Theorem \ref{thm:vsacover}. As $X(k)=\emptyset$, we have $V_i(k)=\emptyset$. By VSA, this implies $V_i(\Ab_{k,S})^{\omega}=\emptyset$. By Lemma \ref{lemma:combined1}(\ref{item:change_places}), we also have $V_i(\Ab_{k,S'})^{\omega}=\emptyset$ for every $i$.
\end{proof}





\section{Embeddings of Varieties into Tori}\label{sec:prod_tori}

For the rest of Section \ref{sec:prod_tori}, we assume that $k$ is a number field.

\begin{lemma}\label{lemma:fabtori}
Let $T$ be an algebraic torus over $k$, let $S$ be a nonempty set of places of $k$, and let $\alpha \in T(\Ab_{k,S})$. Consider the following statements:
\begin{enumerate}
    \item[(i)] $\alpha \in T(\Ab_{k,S})^{\fab}$.
    \item[(ii)] For every $n \in \Zb_{\ge 0}$, there exists $a_n \in T(k)$ and $\beta_n \in T(\Ab_{k,S})$ for which $\alpha = a_n (\beta_n)^n$.
    \item[(iii)] $\alpha \in \cl{T(k)}$, with the closure taken in $T(\Ab_{k,S})$.
\end{enumerate}
Then (i) implies (ii), and if $S$ consists only of finite places of $k$, then (ii) implies (iii).
\end{lemma}
\begin{proof}

(i) $\implies$ (ii): For every $n$, there is a standard torsor $T \xrightarrow{[n]} T$ under the $n$-torsion scheme $T[n]$ of $T$ given by the $n$th power map $T \to T$. The Kummer map is the map sending $x \in T(k)$ to the pullback of $T \xrightarrow{[n]} T$ under $x$, the result of which is a torsor under $T[n]$ over $k$. In this case, it is given by the boundary map $T(k) \to H^1(k,T[n])$ in Galois cohomology coming from the short exact sequence $0 \to T[n] \to T \to T \to 0$ of \'etale sheaves over $\Spec{k}$. In particular, the image of the Kummer map is canonically $T(k)/nT(k)$. The same holds with $k_v$ in place of $k$, and all maps respect the inclusions $k \to k_v$. If $\alpha \in T(\Ab_{k,S})^{\fab}$, it must be in the image of $$T(k)/nT(k) \to \prod_{v \in S} T(k_v)/nT(k_v),$$ which amounts to saying that $\alpha = a_n (\beta_n)^n$ for $\beta_n \in \prod_{v \in S} T(k_v)$. As $a_n,\alpha \in T(\Ab_{k,S})$, so is $\beta_n$.

(ii) $\implies$ (iii): Let $K$ be any open subgroup of $T(\Ab_{k,S})$.
By \cite[Theorem 5.1]{borel63} (also c.f. \cite{conrad06})\footnotemark, $T(\Ab_{k,S})/T(k)K$ is finite, say of order $h$.  We know that $\alpha = a_h (\beta_h)^h$. But $(\beta_h)^h \in T(k)K$, and therefore so is $\alpha$. In other words, the coset $\alpha K$ contains an element of $T(k)$. The set of all such $K$ and their cosets form a basis for the topology of $T(\Ab_{k,S})$ because $S$ contains only finite places. It follows that every open neighborhood of $\alpha$ contains an element of $T(k)$, or $\alpha \in \cl{T(k)}$.
\footnotetext{In fact, \cite{borel63} covers only the case that $\Ab_k = \Ab_{k,S}$, but the result holds in general by considering the continuous projection $\Ab_k \to \Ab_{k,S}$.}


\end{proof}

\begin{lemma}\label{lemma:finite_units_vsa}
Let $T$ be an algebraic torus over $k$ and $K$ an open subgroup of $T(\Ab_{k,S})$ such that $T(k) \cap K$ is finite. Then $T(k)$ is closed in $T(\Ab_{k,S})$. For example, this happens for $S=f$ if $T$ has a model $\Tc$ over $\Oc_k$ with $\Tc(\Oc_k)$ finite.
\end{lemma}


\begin{proof}
As $T(k)$ is a subgroup, each coset of $K$ also has finite intersection with $T(k)$. The topological space $T(\Ab_{k,S})$ is the disjoint union of the cosets of $K$. Since finite sets are closed, $T(k)$ is closed in each coset of $K$. This implies that it is closed in all of $T(\Ab_{k,S})$.

By the definition of the adelic topology, the subgroup $$\Tc(\widehat{\Oc_k})  \colonequals  \prod_v \Tc(\Ov) \subseteq T(\Ab_k^f)$$ is open. We can then conclude because $\Tc(\Oc_k) = \Tc(k) \cap \Tc(\widehat{\Oc_k})$.
\end{proof}

\subsection{The Result When \texorpdfstring{$k$}{k} is Imaginary Quadratic}\label{sec:qb_quad_imag}

In this subsection, we prove $A(\fab,f,k)$ when $k$ has finitely many units. While the result is essentially \cite[Corollary 3.10]{LiuXu15}, we have included it to show how it fits into the general formalism of Definition/Theorem \ref{dfthm:A}.

\begin{rem}By Dirichlet's unit theorem, $k$ has finitely many units if and only if $k$ is $\Qb$ or imaginary quadratic.\end{rem}

\begin{cor}\label{cor:fabsomek}For $k$ as above, $\Gb_m$ is VSA for $(\fab,f,k)$. In particular, $A(\fab,f,k)$ holds.\end{cor}

\begin{proof}
That $k$ has finitely many units means that $\Gb_m(\Oc_k)$ is finite. By Lemma \ref{lemma:finite_units_vsa}, this implies that $\Gb_m(k)$ is closed in $\Gb_m(\Ab_k^f)$. By Lemma \ref{lemma:fabtori}, this implies that $\Gb_m(k) = \Gb_m(\Ab_k^f)^{\fab}$; i.e., $\Gb_m$ is VSA for $(\fab,f,k)$.
\end{proof}

\subsection{The Result When \texorpdfstring{$k$}{k} is Totally Real}

We now prove $A(\fab,f,F)$ when $k$ is totally real. The material in this subsection is inspired by \cite{sx12b}.

In fact, we can choose the subscheme of $\Pb^1_k$ as in Definition/Theorem \ref{dfthm:A}(i) to be the complement in $\Pb^1_k$ of the vanishing scheme of any quadratic polynomial over $k$ with totally negative discriminant.

\begin{defn}For $k$ a totally real number field and $E/k$ a totally imaginary quadratic extension, we define the \emph{norm one torus relative to $E/k$} $$T = T_{E/k} = \ker(\N{E}{k} \colon \res_{E/k}\Gb_m \to \Gb_m),$$ which is a group scheme over $k$.\end{defn}

\begin{prop}\label{prop:embedsnormone}
Let $\alpha \in E$ with minimal polynomial $t^2+bt+c$ over $k$. Let $U$ be the complement in $\Pb^1_k$ (with coordinate $t$) of the vanishing locus of $t^2+bt+c$. Then $U$ is isomorphic to $T$.\end{prop}

\begin{proof}
As $E=k(\alpha)$, and the norm of $x-y\alpha$ is $x^2+bxy+cy^2$, we can express $T$ as $\spec{k[x,y]/(x^2+bxy+cy^2-1)}$. This has projective closure $\proj(k[x,y,z]/(x^2+bxy+cy^2-z^2))$, which is a smooth projective conic. This conic has a point $(x,y,z)=(1,0,1)$, so it is isomorphic to $\Pb^1_k$. The complement of $T$ is given by $0=z^2=x^2+bxy+cy^2$, which corresponds to the vanishing locus of $t^2+bt+c$ by setting $t=x/y$.
\end{proof}

\begin{prop}\label{prop:tisvsa}
For some integral model $\Tc$ of $T$, the set $\Tc(\Oc_k)$ is finite. Thus $T$ is VSA for $(\fab,f,k)$ by Lemma \ref{lemma:finite_units_vsa}.
\end{prop}


\begin{proof}
The torus $T$ has an integral model $\Tc = \ker(\N{\Oc_E}{\Oc_k} \colon \res_{\Oc_E/\Oc_k}\Gb_m \to \Gb_m)$. Thus $\Tc(\Oc_k)$ is the kernel of $\N{\Oc_E}{\Oc_k} \colon \Gb_m(\Oc_E) \to \Gb_m(\Oc_k)$.

The composition $\Gb_m(\Oc_k) \hookrightarrow \Gb_m(\Oc_E) \xrightarrow{\N{\Oc_E}{\Oc_k}} \Gb_m(\Oc_k)$ is $x \mapsto x^2$. Its image in the finitely generated abelian group $\Gb_m(\Oc_k)$ therefore has full rank, and therefore so does the image of $N \colon \Gb_m(\Oc_E) \to \Gb_m(\Oc_k)$. But $E$ and $k$ have the same number of Archimedean places, so $\Gb_m(\Oc_E)$ and $\Gb_m(\Oc_k)$ have the same rank. But a map between finitely generated abelian groups of the same rank whose image has full rank must have finite kernel. Therefore $\Tc(\Oc_k)$ is finite.
\end{proof}

\begin{rem}
One may alternatively prove Proposition \ref{prop:tisvsa} by noting that $\Ok$ is discrete in $k \otimes \Rb$ and that $T(k \otimes \Rb)$ is compact.
\end{rem}

\begin{cor}\label{cor:uisvsa}
If $U$ is the complement in $\Pb^1_k$ of the vanishing locus of a quadratic polynomial $t^2+bt+c$ with totally negative discriminant, then $U$ is VSA for $(\fab,f,k)$. In particular, $A(\fab,f,k)$ holds.
\end{cor}

\begin{proof}
Let $E$ be the splitting field of $t^2+bt+c$. By Proposition \ref{prop:embedsnormone}, $U$ is isomorphic to $T_{E/k}$, so by Proposition \ref{prop:tisvsa}, $U$ is VSA for $(\fab,f,k)$.
\end{proof}

\section{Finite Descent and the Section Conjecture}

We now show how the section conjecture for $\Pb^1_k \setminus \{0,1,\infty\}$ implies $A(\fcov,f,k)$ for any number field $k$.

\subsection{Grothendieck's Section Conjecture}\label{sec:grosec_conj}


Following \cite{sxbook13}, we let $\seck{X}$ denote the set of fundamental group sections up to conjugation, and $\kappa \colon X(k) \to \seck{X}$ the profinite Kummer map, defined for any quasi-compact quasi-separated geometrically connected scheme over $k$. If $k'/k$ is an extension of characteristic $0$ fields, there is a base change map $\seck{X} \to \sections{X}{k'}$ compatible with the Kummer map.


As we often consider non-proper varieties, we need the notion of cuspidal section. For a smooth curve $U/k$, there is a set $C_U$ of cuspidal sections, which we now recall following \cite[Definition 248]{sxbook13}.

\begin{defn}\label{defn:cusp_curve}
Let $U/k$ be a smooth curve and $U \hookrightarrow X$ be its unique smooth compactification. Let $Y=X \setminus U$ and $y \in Y(k)$. We fix a Henselization $\Oc^h_{X,y}$ and set $X_y = \spec{\Oc^h_{X,y}}$. We let $U_y = X_y \times_X U$. Then
$$
C_U \colonequals \bigcup_{y \in Y(k)} \mathrm{Im}\left(\seck{U_y} \to \seck{U}\right) \subseteq \seck{U}.
$$
\end{defn}

\begin{rem}\label{rem:compat_cusp_curve}
Let $U/k$ be as in Definition \ref{defn:cusp_curve}, and let $k'/k$ be an extension of characteristic $0$ fields. Then the base change of a cuspidal section from $k$ to $k'$ is still cuspidal.
\end{rem}

The following collects some known results listed in \cite{sxbook13}:

\begin{prop}\label{prop:inj_sec_conj}
Let $k$ be a subfield of a finite extension of $\Qb_p$ and $U/k$ a geometrically connected, smooth hyperbolic curve or a proper genus $1$ curve. Then 
\begin{enumerate}
    \item\label{item:kappa_inj} The map $\kappa \colon U(k) \to \seck{U}$ is injective.
    \item\label{item:cusp_disjoint} The subsets $\kappa(U(k))$ and $C_U$ of $\seck{U}$ are disjoint.
\end{enumerate}\end{prop}

\begin{proof}
First let $k$ be a finite extension of $\Qp$. The injectivity of the Kummer map follows from Corollary 74 and Proposition 75 of \cite{sxbook13}. The second part follows from \cite[Theorem 250]{sxbook13} for $U$ hyperbolic and is vacuous for $U$ proper.

For the case of general $k$, suppose that $k'$ is a finite extension of $\Qb_p$, and $k \subseteq k'$. As $U(k) \hookrightarrow U(k')$ and $\kappa_{U/k'}$ are injective, $\kappa_{U/k}$ is as well, injectivity holds. The second part follows from Remark \ref{rem:compat_cusp_curve}.
\end{proof}

The section conjecture of Grothendieck is as follows:

\begin{conj}[Grothendieck \cite{Gro83}]\label{c:gro2} Let k be a number field and $U/k$ a geometrically connected, smooth hyperbolic curve (not necessarily proper). Then the image of the Kummer map
\[
\kappa \colon U(k) \to \seck{U}
\]
contains $\seck{U} \setminus C_U$.

the surjectivity in the section conjecture holds for $(C,C_U,k)$.
\end{conj}

\subsection{VSA via the Section Conjecture}

We now explain the relationship between the section conjecture and very strong approximation:

\begin{prop}\label{prop:sec_conj_vsa}
Let $S$ be a nonempty set of finite places of a number field $k$ and $X$ a hyperbolic curve satisfying Conjecture \ref{c:gro2} over $k$. Then $X$ satisfies VSA for $(\fcov,S,k)$
\end{prop}

\begin{proof}
Let $P \in X(\Ab_{k,S})^{\fcov}$, and let $P' \in X(\Ab_{k})^{\fcov}$ project to $P$. For each place $v$ of $k$, let $s_v \colon = \kappa_{X/k_v} (P_v') \in \sections{X}{k_v}$. Then $P'$ satisfies \cite[Theorem 2.1(i)]{harsx12} for $U$ the trivial group and $S$ (in the notation of \cite{harsx12}) the empty set. It therefore satisfies (iii) of the same theorem; i.e., $(s_v)$ is image of some $s \in \seck{X}$ under the bottom horizontal arrow of the following diagram:
$$
\begin{CD}
X(k) @>>> X(\Ab_{k})\\
@V\kappa_{X/k} VV @VV \kappa_{X/\Ab_k} V\\
\seck{X} @>>> \prod_{v} \sections{X}{k_v}
\end{CD}
$$

If $s$ were cuspidal, $s_v$ would be for each $v$ by Remark \ref{rem:compat_cusp_curve}. But $s_v$ is not cuspidal for every finite $v$ by the second part of Proposition \ref{prop:inj_sec_conj}, so $s$ is not cuspidal. By Conjecture \ref{c:gro2}, there is an element $x$ of $X(k)$ mapping to $s$. It has the same image as $P$ under $\prod_{v} \kappa_{X/k_v}$. But Proposition \ref{prop:inj_sec_conj} tells us that $\kappa_S$ is injective, so $P'$ equals $x$ at all $v \in S$. Thus, $P$ equals $x$ and therefore lies in $X(k)$.
\end{proof}

This now allows us to answer Question \ref{quest:vsa}.

\begin{cor}\label{cor:fcovallk}
Conjecture \ref{c:gro2} implies that the property $A(\fcov,S,k)$ holds for any number field $k$ and nonempty set $S$ of finite places.
\end{cor}
\begin{proof}
Proposition \ref{prop:sec_conj_vsa} implies that $\poneminusthreepoints$ is VSA for $(\fcov,S,k)$. By Theorem \ref{dfthm:A}, $A(\fcov,S,k)$ holds.
\end{proof}



\section{Finite Abelian Descent for all \texorpdfstring{$k$}{k}}

Let $k$ be a number field. Lemma \ref{lemma:fabtori} states that the only reason tori do not satisfy VSA for $(\fab,f,k)$ is that the rational points are not closed in the adelic points. The following result therefore suggests that $\poneminusthreepoints$ might play the role of tori over any number field $k$:

\begin{prop}
Let $X=\poneminusthreepoints_k$. The set $X(k)$ is closed in $X(\Ab_k^f)$.
\end{prop}
\begin{proof}
Let $\alpha \in X(\Ab_k^f) \setminus X(k)$. We find a neighborhood of $\alpha$ that does not intersect $X(k)$.

There must be some finite set $S$ of places of $k$ for which $\alpha \in U_S \colonequals \prod_{v \in S} X(k_v) \times \prod_{v \notin S} X(\Ov)$. But $$X(k) \cap U_S=X(\Oc_{k,S})$$ is finite by Siegel's Theorem, so $U_S \setminus X(k)$ is an open set containing $\alpha$.
\end{proof}

This leads us to conjecture:

\begin{conj}\label{conj:faballk}
$\poneminusthreepoints_k$ is VSA for $(\fab,f,k)$, so $A(\fab,f,k)$ for any number field $k$.\end{conj}

In fact, a weaker version of this follows from a conjecture of Harari and Voloch:

\begin{prop}
\cite[Conjecture 2]{harvol09} implies that the answer to Question \ref{quest:emptiness} is yes for the finite abelian descent obstruction over any number field $k$.
\end{prop}

\begin{proof}
By the proof of Theorem \ref{thm:vsacover}, we can decompose any variety into affine open subsets that embed in some power of $(\poneminusthreepoints)_k$. Therefore, it suffices to prove the result for $U/k$ for which there exists $f \colon U \to X^n$ for some positive integer $n$ and $X=(\poneminusthreepoints)_k$. Suppose that $U(\Ab_k)^{\fab}$ is nonempty; we wish to show that $U$ has a $k$-point. 

Let $\alpha \in U(\Ab_k)^{\fab}$. We want to show that each coordinate of $f(\alpha)$ is in $X(\Ab_k)^{\br}$. If $X$ were projective, this would follow from \cite[Corollary 7.3]{sto07}. But $X$ is not, so we must instead prove this using the \'etale homotopy obstruction of \cite{hs13}.

We $f(\alpha) \in X^n(\Ab_k)^{\fab}=X^n(\Ab_k)^{\Zb h,1}$ by \cite[Theorem 9.103]{hs13}. It follows by Lemma \ref{lemma:combined3} that every coordinate $f_i(\alpha)$ ($1 \le i \le n$) of $f(\alpha)$ is in $X(\Ab_k)^{\Zb h,1}$. As $X^s$ is connected, and $H_2(X^s)=0$, we have $X(\Ab_k)^{\Zb h,1} = X(\Ab_k)^{\Zb h}$ by \cite[Theorem 9.152]{hs13}. \cite[Theorem 9.116]{hs13} tells us that $X(\Ab_k)^{\Zb h} = X(\Ab_k)^{\Br}$. Therefore, $f_i(\alpha) \in X(\Ab_k)^{\Br}$. In particular, it pairs to $0$ with every element of $\Br(X)$.

Let $S$ be a finite set of places of $k$ large enough so that $X$ and $f$ are defined over $\Oc_{k,S}$, and $f(\alpha)$ is integral over $\Oc_{k,S}$. Let $\Xc = (\poneminusthreepoints)_{\Oc_{k,S}}$. By forgetting the $v$-components for $v \in S$, we get a point $f_i(\alpha)^S$ of $X(\Ab^S_k)^{\fab} \cap \prod_{v \notin S} \Xc(\Ov)$ for each $i$. The pairing of any element of $B_S(X)$ (in the notation of \cite{harvol09}) with $f_i(\alpha)^S$ is the same as its pairing with $f_i(\alpha)$, for elements of $B_S(X)$ pair trivially with every element of $X(k_v)$; in particular, $f_i(\alpha)^S$ is orthogonal to all of $B_S(X)$. But the conjecture implies that $f_i(\alpha)^S$ comes from an element of $\Xc(\Oc_{k,S})$. It follows that $f(\alpha)_v$ is in $X^n(k)$ for every $v \notin S$. Thus for any $v \notin S$, we find that $f(\alpha)_v \in f(U(k))$.
\end{proof}

\begin{rem}
It does not seem that \cite[Conjecture 2]{harvol09} implies Conjecture \ref{conj:faballk}, for the former requires using only integral points, which in turn means that we must throw out a different set $S$ of primes depending on the point in question.
\end{rem}

\part{The Fibration Method}\label{part:fibration}

In this part, we discuss homotopy fixed points and the \'etale homotopy obstruction in more depth. We use this to analyze the behavior of VSA in fibrations.

.





\section{Homotopy Fixed Points}\label{sec:htpy_sec_conj}

We will now review the technical facts we need about homotopy fixed points. Unless stated otherwise, we let $k$ denote an arbitrary field. Let $\mathcal{S}$ denote the category of simplicial sets.

\begin{defn}The category $\hat{\mathcal{S}}$ of \emph{profinite spaces} is the category of simplicial objects in the category of profinite sets.\end{defn}

\begin{defn}
In \cite[Definition 4.4]{fr82}, Friedlander defines the \'etale topological type $Et(X)$ of a scheme $X$, which is an object of $\Pro(\mathcal{S})$ and represents the \'etale homotopy type of Artin-Mazur (\cite{am69}). This is functorial in the scheme $X$.
\end{defn}

\begin{defn}\label{defn:quickcompletion}
 In \cite[\S 2.7]{quick08}, Quick defines the profinite completion functor $\Pro(\Sc) \to \hat{\Sc}$ and defines $\hat{Et}(X)$ to be the profinite completion of $Et(X)$. The profinite completion of pro-spaces is defined by first taking profinite completion level-wise to get a pro-object in the category of profinite spaces and then taking the inverse limit in the category of profinite spaces.
\end{defn}

The cohomology of this profinite space with finite coefficients recovers the \'etale cohomology of $X$, and the fundamental group of this space is the profinite \'etale fundamental group when $X$ is geometrically unibranch.

\begin{defn}
If $X$ is a $k$-scheme, the \emph{homotopy fixed points} $\hat{Et}(X^s)^{hG_k}$ (\cite[Definition 2.22]{quick11}) are defined as follows. If $G$ is a profinite group, one may use the bar construction to define $EG$ and $BG$ as profinite spaces. Then $\hat{Et}(X^s)^{hG_k}$ is the space of $G_k$-equivariant maps from $EG_k$ to $\hat{Et}(X^s)$. (Compare with \cite[9.3.2]{hs13}, \cite[8.2]{barsch16}, and the introduction to \cite{pal15}). By the discussion immediately following Proposition 2.14 in \cite{quick08}, this has the structure of a simplicial set. As is $\hat{Et}(X^s)$, this is functorial in the $k$-scheme $X$.\end{defn}

\begin{defn}As in \cite{hs13}, \cite{barsch16}, and \cite{pal15}, we denote $\pi_0(\hat{Et}(X^s)^{hG_k})$ by $X(hk)$ and refer to it as the \emph{homotopy sections} of $X/k$. This is again functorial in the $k$-scheme $X$.
\end{defn}


If $k'/k$ is a field extension and $X$ a $k$-scheme, $X(hk')$ still denotes $\pi_0(\hat{Et}(X^s)^{hG_{k'}})$.

\begin{defn}\label{defn:hseck}
The \emph{homotopy profinite Kummer map} $\kappa_{X/k} \colon X(k) \to \hseck{X}$ of \cite[\S 3.2]{quick11} is defined as follows. The set $\spec{k}(hk)$ has one element, so by functoriality, an element of $X(k)$ gives rise to an element of $X(hk)$. Compare with \cite[\S 9.3.2]{hs13}, \cite[8.2]{barsch16}, and the introduction to \cite{pal15}.
\end{defn}

\begin{rem}
We write $\kappa_X$ or even $\kappa$ when $k$, $X$ are understood. We also write $\kappa_S$, $\kappa_{X/\Ab_{k,S}}$, and $\kappa_{X/\Ab_k}$, as in the case of fundamental group sections.
\end{rem}



\subsection{Basepoints and Homotopy Groups}

While using the whole homotopy type helps avoid the use of basepoints, we nonetheless sometimes need to use them. In particular, in order to discuss homotopy groups, we must discuss basepoints for $\hat{Et}(X)$.

On p.37 of \cite{fr82}, it is stated that a geometric point $a$ of $X$ gives rise to a basepoint of $Et(X)$, and this then gives rise to a basepoint of $\hat{Et}(X)$, which we call $b_a$.

If this geometric point is a rational point, it is invariant under $G_k$, so our basepoint of $\hat{Et}(X)$ is invariant under $G_k$, i.e. we have a pointed $G_k$-space. Sometimes $X$ does not have any rational points, and yet we want a $G_k$-invariant basepoint. This is provided by the following construction:

\begin{constr}\label{constr:gkinvbp}
Let $a \in X(hk)$. This is a homotopy class of $G_k$-equivariant maps $EG_k \to \hat{Et}(X^s)$. By taking a homotopy pushout of this map along the map from $EG_k$ to a point, we obtained a new model of $\hat{Et}(X^s)$ as a profinite space with $G_k$-action with a $G_k$-invariant basepoint. Any such basepoint is denoted by $b_a$. 
\end{constr}

\begin{defn}\label{defn:prof_htpy_groups}
In \cite[Definition 2.15]{quick08} and \cite[Definition 2.12]{quick11}, Quick defines the homotopy groups of a pointed profinite space. If $X$ is a scheme, and $a$ is a geometric point or a homotopy section of $X$, we denote $\pi_i(\hat{Et}(X),b_a)$ by $\pi_i^{\acute{e}t}(X,a)$ or simply $\pi_i(X,a)$ (c.f. \hyperref[noteconv]{Notation and Conventions}). If $a$ is a Galois-invariant geometric point or a homotopy section of a scheme $X/k$, this basepoint is fixed by the action of $G_k$, so there is a natural action of $G_k$ on $\pi_i(X^s,a)$.
\end{defn}

\begin{rem}\label{rem:prof_htpy_groups}
Suppose that $X$ is geometrically unibranch (e.g., normal), Noetherian, and connected. By \cite[Proposition 2.33]{quick08}, the homotopy groups of $\hat{Et}(X)$ are isomorphic (as pro-groups) to the pro-homotopy groups of the profinite completion of the Artin-Mazur homotopy type of $X$. By \cite[Theorem 11.2]{am69}, this profinite completion is already isomorphic to the (non-completed) Artin-Mazur homotopy type as a pro-homotopy type. But $Et(X)$ represents the Artin-Mazur homotopy type (c.f. \cite[Corollary 6.3]{fr82}), so its homotopy groups are isomorphic to all of the above. 

More generally, if $X$ is not assumed to be connected (but still Noetherian), then it is a finite disjoint union of connected schemes. Thus the same result holds, as all constructions behave well with respect to finite coproducts.
\end{rem}

\begin{rem}\label{rem:basepoints}
If $X$ is connected, then $\pi_0(\hat{Et}(X))$ contains one element, so the homotopy groups are independent of the choice of basepoint up to isomorphism. This isomorphism is, however, unique only modulo the action of the fundamental group. This is also a subtlety in using different models for $\hat{Et}(X^s)$, as in Construction \ref{constr:gkinvbp}. Furthermore, any two basepoints of $\hat{Et}(X^s)$ arising from the same element of $X(hk)$ are homotopic in the $G_k$-equivariant category, so the resulting actions on $\pi_i(X^s)$ are equivalent. This latter fact is important in the statement of Theorem \ref{thm:dss}. The reader may check that this is does not pose a problem in any of the theorems we prove.\end{rem}

\begin{defn} A $k$-variety $X$ has the \emph{\'etale $K(\pi,1)$-property} or \emph{is an \'etale $K(\pi,1)$} if $X$ is geometrically connected, and $\pi_i(X^s,a) = 0$ for $i \ge 2$ and some (equivalently, any) geometric point $a$.\end{defn}


\begin{prop}\label{prop:etale_kpi1}
Let $X$ be a geometrically connected variety over $k$. Then there is a map $\Trunc_1 \colon X(hk) \to \seck{X}$ (functorial in $X$) whose composition with the homotopy profinite Kummer map is the ordinary profinite Kummer map. Furthermore, if $X$ is an \'etale $K(\pi,1)$ (e.g., by \cite[Lemma 2.7(a)]{schsx16}, when $X$ is a smooth curve not isomorphic to $\Pb^1$), the map $\Trunc_1$ is a bijection.

In particular, if $X$ is a geometrically connected, smooth hyperbolic curve over a number field $k$, then Conjecture \ref{c:gro2} implies that the image of the homotopy Kummer map $\kappa \colon X(k) \to X(hk)$ contains $X(hk) \setminus C_X$, and the homotopy Kummer map is injective.

\end{prop}


\begin{rem}

Related facts are discussed in \cite[\S 3.2]{quick11}, \cite[Proposition 12.6(b)]{pal15}, and \cite[\S 2.6]{sxbook13}, especially \cite[Equation (2.13)]{sxbook13}.
\end{rem}

Before proving Proposition \ref{prop:etale_kpi1}, we must introduce the descent spectral sequence for homotopy fixed points. It is essentially \cite[Theorem 2.16]{quick11} for the pointed profinite $G_k$-space $\hat{Et}(X^s)$:

\begin{thm}[Theorem 2.16 of \cite{quick11}]\label{thm:dss}
For each $a \in X(hk)$, there is a spectral sequence $$E^{p,q}_2 = H^p(G_k;\pi_q(X^s,a)) \Rightarrow \pi_{q-p}(\hat{Et}(X^s)^{hG_k},a),$$ known as the \emph{descent spectral sequence}. The basepoint $a$ of $\hat{Et}(X^s)^{hG_k}$ naturally results from the $G_k$-invariant basepoint $b_a$ of Construction \ref{constr:gkinvbp}, and the $G_k$-action on $\pi_q(X^s,a)$ is well-defined by Remark \ref{rem:basepoints}.
\end{thm}






\begin{proof}[Proof of Proposition \ref{prop:etale_kpi1}]
If $X(hk)$ is empty, then the map $\Trunc_1$ automatically exists, and it is functorial because there is a unique map from an empty set. It follows that $\Trunc_1$ is automatically a bijection. In this case, $X(k)$ and $C$ are empty, so the surjectivity and injectivity in either version of the section conjecture always hold.

We now suppose that there is a homotopy fixed point $a \in X(hk)$. As $\pi_0(X_s,a)$ and hence $H^0(G_k;\pi_0(X^s,a))$ is trivial, the spectral sequence produces a map $X(hk) \to \seck{X}$.

When $X=\spec{k}$, this map is clearly an isomorphism. Functoriality of the spectral sequence and the definition of the Kummer map imply that for arbitrary $X$, the map $X(hk) \to \seck{X}$ respects the Kummer map.

We now suppose that $X$ is an \'{e}tale $K(\pi,1)$. As $\pi_i(X^s)$ contains one element for $i \neq 1$, the spectral sequence tells us that $\pi_{1-p}(\hat{Et}(X^s)^{hG_k},a) \cong H^{p}(G_k;\pi_1(X^s,a))$. For $p=1$, this tells us that $\Trunc_1$ is a bijection. (Alternatively, we could avoid the spectral sequence by noting that an \'etale $K(\pi,1)$ is the classifying space of its fundamental group and then applying \cite[Proposition 2.9]{quick15}.)

The case of a geometrically connected, smooth hyperbolic curve follows by \cite[Lemma 2.7(a)]{schsx16} and Proposition \ref{prop:inj_sec_conj}.
\end{proof}

We record a related lemma that will be used for Theorem \ref{thm:vsa_geom_fib}:

\begin{lemma}\label{lemma:section_centralizer}
Let $U$ be a smooth geometrically connected curve not isomorphic to $\Pb^1$. Then for any $a \in X(hk)$, we have $\pi_1(\hat{Et}(X^s)^{hG_k},a)=0$.
\end{lemma}
\begin{proof}
By \cite[Lemma 2.7(a)]{schsx16}, the descent spectral sequence tells us that $\pi_1(\hat{Et}(X^s)^{hG_k},a) \cong H^{0}(G_k;\pi_1(X^s,a))$. The group $\pi_1(X^s,a)$ has $G_k$-action arising from the section $a$, so $H^{0}(G_k;\pi_1(X^s,a))$ is simply the centralizer of $\Trunc_1(a)$. The result then follows by \cite[Proposition 104]{sxbook13}.
\end{proof}

\subsection{Obstructions to the Hasse Principle}\label{sec:hobstructions}

\begin{prop}\label{prop:hsecbasechange}
If $k'/k$ is an extension of fields of characteristic $0$, and $X$ is a proper or smooth variety over $k$, then there is a base change map $$\hseck{X} \to \hsections{X}{k'}$$
Furthermore, the diagram
$$
\begin{CD}
X(k) @>>> X(k')\\
@V\kappa_{X/k} VV @VV\kappa_{X/k'} V\\
\hseck{X} @>>> \hsections{X}{k'}
\end{CD}
$$
is commutative.
\end{prop}

\begin{proof}
When $X$ is smooth and connected, \cite[Proposition 5.4]{pal15} says that algebraically closed extension of base field in characteristic $0$ induces an isomorphism on \'etale homotopy types. The same follows in the disconnected case by considering a finite disjoint union of connected smooth varieties. The same is also true in the proper case by \cite[Corollary 12.12]{am69}. The result is then clear because a map of groups gives a map on homotopy fixed points.

\end{proof}

Let $X$ be a scheme over a field $k$. Then $\hat{Et}(X^s)$ can be represented as a pro-space $\{X_{\alpha}\}_\alpha$ with $G_k$-action, with each $X_{\alpha}$ $\pi$-finite (c.f. the proof of Lemma \ref{hlemma:fibration}). As $\hat{Et}(X^s) = \underset{\alpha}{\mathrm{holim}} X_{\alpha}$, and homotopy limits commute with homotopy fixed points, we have $X(hk) = \pi_0(\underset{\alpha}{\mathrm{holim}} X_{\alpha}^{h G_k})$.


\begin{defn}\label{defn:homotopy_loc}
Let $X$ be a variety over a number field $k$ for which Proposition \ref{prop:hsecbasechange} applies (i.e., if $X$ is smooth or proper). Then we define
\[
\loc \colon X(hk) \to \prod_v X(hk_v)
\]
as the product over all places $v$ of $k$ of the base change map $X(hk) \to X(hk_v)$.
\end{defn}

We now recall that for a smooth variety $X$ over a number field $k$ there is an \'etale homotopy obstruction to rational points defined on p.314 of \cite{hs13}, denoted by $X(\Ab_k)^h$. These are defined using another version of homotopy fixed points (\cite{hs13}, Definition 3.3), which we denote by $X(hk)_{2}$. For a pro-space $\{X_\alpha\}$ with $G_k$-action, it is given by taking $\lim_{\alpha} \pi_0(X_{\alpha}^{hG_k})$ (while \cite{hs13} uses the $X_{\alpha}^{hG_k}$ of \cite{goerss95}, it is easy to see that this is the same as that of \cite{quick11} by considering the descent spectral sequence). There is a natural map $X(hk) \to X(hk)_{2}$ giving us a diagram

\centerline{
\xymatrix{X(k) \ar[r] \ar[d]_{\kappa_{X/k}} & X(\Ab_k) \ar[d]^{\kappa_{X/\Ab_k}}\\
X(hk) \ar[d] \ar[r]^-{\loc} & \prod_v X(hk_v) \ar[d] \\
X(hk)_{2} \ar[r]^-{\loc} & \prod_v X(hk_v)_{2},
}
}

Then the $X(\Ab_k)^h$ of \cite{hs13} is defined as the subset of the upper-right object of this diagram whose image in the lower-right object is contained in the image of the bottom horizontal arrow.



\begin{lemma}\label{lemma:quickhscompare}
The map $X(hk) \to X(hk)_{2}$ is surjective, and the map $X(hk_v) \to X(hk_v)_{2}$ is an isomorphism for each $v$.
\end{lemma}

\begin{proof}
As $X$ is defined over a number field, we can express $\hat{Et}(X^s)$ as a pro-space $\{X_{\alpha}\}$, where $\alpha$ ranges over a countable cofiltering category. It follows that we may replace it by a tower, so we may apply Proposition \cite[Proposition VI.2.15]{goerssjardine09} to get an exact sequence
$$
\ast \to {\lim_{\alpha}{}^1} \, \pi_1 X_{\alpha}^{h G_k} \to X(hk) \to X(hk)_{2} \to \ast
$$

This implies the desired surjectivity.

Replacing $k$ by $k_v$, we note that we can assume (by \cite[Theorem 11.2]{am69}) that each $X_{\alpha}$ has finite homotopy groups. As $G_{k_v}$ is topologically finitely generated, this implies by Theorem \ref{thm:dss} that $\pi_1 X_{\alpha}^{h G_k}$ is finite for each $\alpha$. But this implies that our ${\lim_{\alpha}{}^1}$ vanishes, so we get the desired isomorphism.
\end{proof}

\begin{prop}\label{prop:homotopy_obstruction_comparison}
In the notation of the previous diagram, $X(\Ab_k)^h$ is $\kappa_{X/\Ab_k}^{-1}(\mathrm{Im}(\loc))$.
\end{prop}

\begin{proof}
This follows immediately from Lemma \ref{lemma:quickhscompare} and the definition of $X(\Ab_k)^h$.
\end{proof}

\section{Homotopy Fixed Points in Fibrations}\label{hsec:sec_conj_fib}

\begin{defn}[Definition 1.1 of \cite{fr82}]\label{defn:geom_fib}
A \emph{special geometric fibration} is a morphism $f \colon X \to S$ of schemes fitting into a diagram:
$$
\xymatrix{
X  \ar@{^{(}->}[r]^{j} \ar[rd]^f & \overline{X} \ar[d]^{\overline{f}} & Y\ar@{_{(}->}[l]^{i} \ar[ld]^g \\
\empty & S &\empty
}
$$
satisfying the following conditions:
\begin{enumerate}
\item $i$ is a closed embedding
\item $j$ is a open immersion which is dense in every fiber of $\overline{f}$, and $X = \overline{X} - Y$
\item $\bar{f}$ is smooth and proper
\item  $Y$ is a union of schemes $Y_1,\cdots,Y_m$, with $Y_i$ of pure relative codimension $c_i$ in $\cl{X}$ over $S$, with the property that every intersection $Y_{i_1} \cap \cdots \cap Y_{i_k}$ is smooth over $S$ of pure codimension $c_{i_1} + \cdots + c_{i_k}$.
\end{enumerate}

More generally, a \emph{geometric fibration} is a map $f \colon X \to S$ admitting a Zariski covering $\{V_i \to S\}$ such that $f_{V_i} \colon X \times_S V_i \to V_i$ is a special geometric fibration for all $i$.

\end{defn}

\begin{lemma}\label{hlemma:gf_pullback}
Let $f:X \to S$ be a geometric fibration  and $U \to S$ a morphism of schemes. Then the base change
$$f_U :X \times_S U \to U $$ is a geometric fibration.
\end{lemma}
\begin{proof}

This is clear because all of the properties in Definition \ref{defn:geom_fib} are stable under base extension.
\end{proof}

\begin{lemma}\label{hlemma:fibration}
Let $f \colon X \to S$ be a map of $k$-schemes. Suppose that $f \colon X \to S$ is a geometric fibration of Noetherian normal schemes with connected geometric fibers, and $k$ has characteristic $0$. Then the sequence $$\hat{Et}(X_s^s) \to \hat{Et}(X^s) \to \hat{Et}(S^s)$$ is a fibration sequence of profinite etale homotopy types (in the model structure of \cite{quick08}).
\end{lemma}

\begin{proof}
We recall some notation from \cite{barhaho17}. We let $\Sc$ denote the category of small simplicial sets with the Kan-Quillen model structure (known as the category of \emph{spaces}) and $\Pro(\Sc)$ the pro-category, with the induced model structure as described in \cite[Theorem 4.2.4]{barhaho17}, which means in particular that the projective limit defining an object is also a homotopy limit. We denote by $\Pro(\Sc)_\infty$ the associated $\infty$-category. A space $X$ is $\pi$-finite if it has finitely many connected components, and for each $x \in X$, the homotopy groups $\pi_n(X,x)$ are all finite and vanish for sufficiently large $n$. The model category $L_{K^\pi} \Pro(\Sc)$ is a certain left Bousfield localization of $\Pro(\Sc)$ in which all $\pi$-finite spaces are fibrant (\cite[Proposition 7.2.10]{barhaho17}). The category $\Sc_\tau$ is the category of all spaces that are levelwise finite and $n$-coskeletal for some $n$, and the natural inclusion $\iota \colon \Sc_\tau \to \hat{\Sc}$ induces an equivalence of categories $\Pro(\Sc_\tau) \to \hat{\Sc}$ by \cite[Theorem 7.4.1]{barhaho17}. We use this equivalence to transfer Quick's model structure from $\hat{\Sc}$ to $\Pro(\Sc_\tau)$. The functor $\Psi \colon \Pro(\Sc) \to \Pro(\Sc_\tau)$ is the left adjoint of the inclusion $\Phi \colon \Pro(\Sc_\tau) \to \Pro(\Sc)$, and the functor $\Psi_{K_{\pi}}$ is the same underlying functor as $\Psi$ but with the $L_{K^\pi}$ model structure on the target.

It follows by \cite[Theorem 11.5]{fr82} (also c.f. \cite[Theorem 3.7]{fr73a}) that the sequence $\Et(X_s) \to \Et(X) \to \Et(S)$ is a fibration sequence, or equivalently that

\centerline{
\xymatrix{
\Et(X_s) \ar[d] \ar[r] &\Et(X)\ar[d]\\
\ast \ar[r] &\Et(S)}
}

is a homotopy pullback diagram. We need to show that the Quick profinite completion (c.f. Definition \ref{defn:quickcompletion}) of this homotopy pullback diagram remains a homotopy pullback diagram.



We may assume this is a diagram of fibrant cofibrant objects in $\Pro(\Sc)$, with the Isaksen model structure. For an object $X$ of $\Pro(\Sc)$, let $X^{\natural}$ denote its replacement by its Postnikov tower. Then the map $X \to X^{\natural}$ induces an isomorphism on all homotopy groups, so the same is true after profinite completion by Remark \ref{rem:prof_htpy_groups}. \cite[Proposition 2.34]{quick08} shows that the profinite completion of this map induces an isomorphism on cohomology for all profinite coefficient systems, hence by \cite[Definition 2.6]{quick08} it is a weak equivalence.

By \cite[Theorem 11.2]{am69}, each is isomorphic to a levelwise $\pi$-finite space (after possibly replacing each by its Postnikov tower, which we have shown is invariant under Quick's profinite completion). As a limit of fibrant objects, it is itself fibrant, i.e., this is a diagram of objects fibrant for $L_{K^\pi} \Pro(\Sc)$. In particular, this means that it is in the image of the inclusion $(L_{K^\pi} \Pro(\Sc))_\infty \hookrightarrow \Pro(\Sc)_\infty$ of infinity categories which is induced by the identity functor on model categories. This inclusion of infinity categories is conservative and preserves limits, so it detects limits. In particular, the above diagram is a pullback diagram in $(L_{K_\pi} \Pro(\Sc))_\infty$, or equivalently, a homotopy pullback diagram in the model category $L_{K_\pi} \Pro(S)$. Therefore, by \cite[Theorem 7.4.8]{barhaho17}, applying $\Psi_{K_{\pi}}$ to this diagram yields a homotopy pullback diagram in $\Pro(\Sc_{\tau}) \cong \hat{\Sc}$.


We are done if we can identify the map $\Psi = \Psi_{K^\pi}$ (these are the same functor on ordinary categories, just refer to different model structures) with the profinite completion functor of Definition \ref{defn:quickcompletion}. There is an adjunction $\adjunction{\Psi}{\Pro{\Sc}}{\hat{\Sc}}{\Phi}$ as well as an adjunction $\adjunction{\mathrm{const}}{\Sc}{\Pro(\Sc)}{\lim}$. By construction, $\lim \circ \Phi$ takes a profinite space $X$ to an object of $\Pro(\Sc_\tau)$ whose limit is $X$, then takes the limit of this pro-object as a pro-space. In particular, this is the forgetful functor from profinite spaces to simplicial sets of p.587 of \cite{quick08}, and therefore $\Psi \circ \mathrm{const}$ is its left adjoint, i.e., Quick's profinite completion functor from $\Sc$ to $\hat{\Sc}$.

It remains to show that $\Psi$ can be obtained by applying $\Psi$ levelwise and taking the limit. We therefore wish to show that $\Psi$ preserves filtered limits. For this, let $\{X_j\}$ be a filtered system in $\Pro(\Sc)$. For $T$ an object of $\Sc_\tau$, we have $\homo( \lim_j \Psi( X_j),T) = \colim_j \homo(\Psi(X_j),T)$ because $T$ is cocompact. But this latter is $\colim_j \homo(X_j,\Phi(T))$, which in turn is $\homo(\lim_j X_j, \Phi(T)) = \homo(\Psi(\lim_j X_j),T)$ because $\phi(T)$ is in $\Sc$ and therefore cocompact. As $\Sc_\tau$ cogenerates $\Pro(\Sc_\tau)$, this shows that $\lim_j \Psi(X_j) = \Psi(\lim_j X_j)$ in $\Pro(\Sc_\tau)$, and we are done.
\end{proof}

\begin{lemma}\label{hlemma:les}
Let $f \colon X \to S$ be a map of $k$-schemes. Let $a \in X(hk)$ and $s \in S(k)$ such that $f(a) = \kappa(s)$. Let $X_s$ be the fiber of $f$ above $s$. Suppose that the sequence $\hat{Et}(X_s^s) \to \hat{Et}(X^s) \to \hat{Et}(S^s)$ is a fibration sequence of profinite etale homotopy types (e.g., by Lemma \ref{hlemma:gf_pullback} and Lemma \ref{hlemma:fibration}, if $f$ is a geometric fibration of Noetherian normal schemes with connected geometric fibers).

Then $a$ comes from an element of $X_s(hk)$, also denoted by $a$, and there is a long exact sequence $$\pi_1(\Et(X_s^s)^{hG_k},a) \to \pi_1(\Et(X^s)^{hG_k},a) \to\pi_1(\Et(S^s)^{hG_k},\kappa(s)) \to X_s(hk) \to X(hk) \to S(hk),$$
where the latter three are pointed sets based at $a$ and $\kappa(s)$.
\end{lemma}

\begin{proof}

As homotopy fixed points commutes with homotopy limits, the sequence $\hat{Et}(X_s^s)^{hG_k} \to \hat{Et}(X^s)^{hG_k} \to \hat{Et}(S^s)^{hG_k}$ is a fibration sequence of simplicial sets. This implies that we get an element of $X_s(hk)$ mapping to $a$.

We then choose a basepoint $b_a$ of $\hat{Et}(X_s^s)^{hG_k}$ in the connected component represented by $a \in X_s(hk)$, which makes $\hat{Et}(X_s^s)^{hG_k} \to \hat{Et}(X^s)^{hG_k} \to \hat{Et}(S^s)^{hG_k}$ a fibration sequence of pointed simplicial sets. The exact sequence then follows by the long exact sequence of homotopy groups for a pointed fibration.
\end{proof}

We finally prove a general result about VSA for the \'etale homotopy obstruction that will be useful in Part \ref{part:examples}. Before proving it, we need one definition.

\begin{defn}[Notation 12.5 of \cite{pal15}]\label{defn:selmer}
If $X$ is a scheme over a number field $k$ to which Definition \ref{defn:homotopy_loc} applies, the Selmer set $\Sel(X/k)$ is the subset of $X(hk)$ whose image under $\loc$ is in the image of $X(\Ab_k)$ in the following diagram:

$$
\begin{CD}
X(k) @>>> X(\Ab_k)\\
@VVV @VVV\\
X(hk) @>\loc >> \prod_v X(hk_v),
\end{CD}
$$
Clearly $\kappa(X(k)) \subseteq \Sel(X/k)$.
\end{defn}


\begin{lemma}\label{cor:HSP_curves}
If $X$ is a geometrically connected, smooth hyperbolic curve over a number field $k$, and Conjecture \ref{c:gro2} holds for $X$, then $\kappa(X(k))=\Sel(X/k)$.
\end{lemma}
\begin{proof}
As mentioned in Definition \ref{defn:selmer}, $\kappa(X(k)) \subseteq \Sel(X/k)$. Now suppose $s \in \Sel(X/k)$. It suffices by Conjecture \ref{c:gro2} to show that $s$ is not cuspidal. But if it were cuspidal, its base change to $k_v$ would be cuspidal (by Remark \ref{rem:compat_cusp_curve}), contradicting (\ref{item:cusp_disjoint}) of Proposition \ref{prop:inj_sec_conj} over $k_v$, so it is not.\end{proof}

\begin{thm}\label{thm:vsa_geom_fib}
Let $f \colon X \to S$ be a geometric fibration of smooth varieties over a number field $k$ with connected geometric fibers, let $T$ be a nonempty set of places of $k$. Suppose that the following conditions hold:


\begin{enumerate}
    \item\label{item:vsa_base} $S$ is VSA for $(h,T,k)$.

\item One of the following is true:

\begin{enumerate}

\item\label{item:sec_conj} $S$ is a geometrically connected, smooth hyperbolic curve over $k$ for which Conjecture \ref{c:gro2} holds

\item\label{item:img_sel} The image of $S(k)$ in $S(hk)$ under $\kappa$ is $\Sel(S/k)$, and $\kappa_{S/k_v} \colon S(k_v) \to S(hk_v)$ is injective for some $v \in T$.

\item\label{item:inj_sel} The localization map $\loc_T \colon S(hk) \to \prod_{v \in T} S(hk_v)$ is injective on $\Sel(S/k)$.
\end{enumerate}

\item\label{item:finite_inj} For all finite $v \notin T$, $\kappa_{S/k_v}$ is injective (vacuous if $T=f$).



\item\label{item:tot_imag} For all finite $v$, every connected component of $\Et(S^s)^{hG_{k_v}}$ in the image of the homotopy Kummer map is simply connected (e.g., by Lemma \ref{lemma:section_centralizer}, if $S$ is a smooth geometrically connected curve not isomorphic to $\Pb^1$).




\item\label{item:real_places} For every real place $v$ of $k$, every $a \in S(k)$, and every $b \in X_a(k_v)$, the map $\pi_1(X(k_v),b) \to \pi_1(S(k_v),a)$ is surjective and $\pi_1(X_a((k_v)_s))$ is trivial (where $\pi_1$ denotes the topological fundamental group under the $v$-adic topology).

\item\label{item:vsa_fiber} For every $a \in S(k)$, the fiber $X_a$ of $f$ above $a$ is VSA for $(h,T,k)$.

\end{enumerate}

Then $X$ is VSA for $(h,T,k)$. One may replace $(h,T,k)$ by $(\mathrm{\acute{e}t-Br},T,k)$ whenever the variety in question is geometrically connected.
\end{thm}

\begin{proof}
We first note that \ref{item:sec_conj} implies \ref{item:img_sel}, which implies \ref{item:inj_sel}. The first implication is Lemma \ref{cor:HSP_curves} and Proposition \ref{prop:inj_sec_conj}. For the second implication, suppose $\alpha,\beta \in \Sel(S/k)$ with $\loc_v(\alpha)=\loc_v(\beta)$. Then by \ref{item:img_sel}, there are $a,b \in S(k)$ such that $\alpha=\kappa(a)$, and $\beta=\kappa(b)$. Therefore $a,b \in S(k_v)$ have the same image under $\kappa_{X/k_v}$, so injectivity of $\kappa_{X/k_v}$ tells us that $a=b$. This implies that $\alpha=\beta$, so \ref{item:inj_sel} holds.

We recall that by Proposition \ref{prop:homotopy_obstruction_comparison}, $X(\Ab_k)^h = \kappa_{X/\Ab_k}^{-1}(\mathrm{Im}(\loc))$.

Now let $\alpha \in X(\Ab_k)^h$ project to $\alpha' \in X(\Ab_{k,S})$, and let $\beta \in X(hk)$ such that $\loc(\beta)=\kappa(\alpha)$. Then $\kappa(f(\alpha)) = \loc(f(\beta))$, so $f(\alpha) \in S(\Ab_k)^h$. By \ref{item:vsa_base}, there is $a \in S(k)$ such that $a=f(\alpha)_v$ for all $v \in T$.

For $v \in T$, we have $\loc_v(f(\beta))=\kappa_{S/k_v}(f(\alpha)_v)=\kappa_{S/k_v}(a)=\loc_v(\kappa_{S/k}(a))$. Thus $\loc_T(f(\beta)) = \loc_T(\kappa_{S/k}(a))$. It's clear that $f(\beta),\kappa_{S/k}(a) \in \Sel(S/k)$, so \ref{item:inj_sel} implies that $f(\beta)=\kappa_{S/k}(a)$.

This implies by Lemma \ref{hlemma:les} that there exists $\gamma \in X_a(hk)$ such that $i_a(\gamma)=\beta$, where $i_a$ denotes the inclusion $X_a \hookrightarrow X$.

We now know that $\kappa_{S/k_v}(f(\alpha)_v) = \loc_v(\beta)=\kappa_{S/k_v}(a)$ for all $v$. By \ref{item:finite_inj}, we know that $a=f(\alpha)_v$ for all finite $v$. For $v$ infinite, \cite[Theorem 1.2]{pal15} implies that $a$ and $f(\alpha)_v$ are in the same connected component of $S(k_v)$. As $f$ is an elementary fibration and therefore a fiber bundle for the $v$-adic topology, we can modify $\alpha$ so that $f(\alpha)_v=a$ and $\kappa(\alpha)$ remains the same. We do this at all infinite places, so that $f(\alpha)=a$, yet still $\kappa(\alpha)=\loc(\beta)$.

We now have $\alpha \in X_a(\Ab_k)$ such that $i_a(\loc(\gamma))=i_a(\kappa_{X_a/\Ab_k}(\alpha))$. We need to show that $i_a:X_a(hk_v) \to X(hk_v)$ is injective for all $v$.

By Lemma \ref{hlemma:les}, \ref{item:tot_imag} implies injectivity for all finite $v$. At all complex $v$, the geometric connectedness assumption on $X_a$ implies that $X_a(hk_v)$ is a point, so injectivity follows.


Let $v$ be a real place. By \cite[Theorem 1.2]{pal10}, the map $X_a(k_v)_{\bullet} \to X_a(hk_v)$ is a bijection, and by \cite[Theorem 1.2]{pal15}, the map $X(k_v)_{\bullet} \to X(hk_v)$ is injective. It therefore suffices to prove that $X_a(k_v)_{\bullet} \to X(k_v)_{\bullet}$ is injective. But $X_a(k_v) \to X(k_v) \to S(k_v)$ is a fibration sequence of real analytic varieties, so the condition on topological fundamental groups guarantees the desired injectivity.


We conclude that $\loc(\gamma)=\kappa_{X_a/\Ab_k}(\alpha)$. We conclude by \ref{item:vsa_fiber} that there is $c \in X_a(k)$ such that $c = \alpha'$.

\end{proof}

\section{Section Conjecture in Fibrations}\label{newsec:sec_conj_fib}

The section conjecture has been reformulated in homotopy-theoretic terms; see \cite[\S 3.2]{quick11} and \cite[Theorem 9.7(b)]{pal15}, or \cite[\S 2.6]{sxbook13} for a summary. However, these sources lack an analogue of cuspidal sections, without which one can restrict only to projective curves or conjecture something like the ``homotopy section property'' (HSP) of \cite{pal15}, a local-to-global statement. We will fix this gap. In fact, HSP follows from our versions of the section conjecture (although the former is expected to hold more generally), as proven for curves in Lemma \ref{cor:HSP_curves}.

The notion of cuspidal sections depends on a compactification. For curves, a smooth compactification is unique. For higher-dimensional varieties, this is not the case, so we make the following definition:

\begin{defn}
A \emph{homotopy cuspidal datum} $(X,C,k)$ is a smooth variety $X$ over $k$ and a subset $C \subseteq \hseck{X}$, called the set of \emph{cuspidal fixed points}.
\end{defn}

\begin{rem}\label{rem:cusp_data_curves}
Whenever $X$ is a smooth curve not isomorphic to $\Pb^1$, we know by Proposition \ref{prop:etale_kpi1} that $X(hk)=\seck{X}$, and we take $C=C_X$ as in Definition \ref{defn:cusp_curve}.
\end{rem}

\begin{rem}\label{rem:higher_dim_cusps}
In practice (e.g. to reprove Corollary \ref{cor:fcovallk}), we will use cuspidal fixed points only for varieties of dimension $>1$ when they arise from good neighborhoods as in Definition \ref{defn:cusp_framing}. However, in \cite{HigherCusp}, we begin the study of a general notion of (non-homotopy) cuspidal sections for higher dimensional varieties.
\end{rem}


\begin{defn}\label{defn:surj_hsec_conj}

A homotopy cuspidal datum $(X,C,k)$ is said to satisfy the \emph{surjectivity in the homotopy section conjecture} if $\hseck{X} \setminus C \subseteq \kappa(X(k))$.

\end{defn}

\begin{rem}\label{rem:surj_ordinary_htpy}
By Proposition \ref{prop:etale_kpi1}, Conjecture \ref{c:gro2} is equivalent to the surjectivity in the homotopy section conjecture for a smooth geometrically connected hyperbolic curve.
\end{rem}

\begin{defn}\label{defn:inj_hsec_conj}

A homotopy cuspidal datum $(X,C,k)$ is said to satisfy the \emph{injectivity in the homotopy section conjecture} if

\begin{enumerate}
\item\label{itm:kummer_inj} The map $\kappa \colon X(k) \to \hseck{X}$ is injective.
\item\label{itm:cusp_faithful} The sets $\kappa(X(k))$ and $C$ are disjoint.
\item\label{itm:simply_conn} Every connected component of $\Et(X^s)^{hG_k}$ in the image of the Kummer map is simply connected.
\end{enumerate}

\end{defn}

We will see in the proof of Theorem \ref{hthm:fib_inj} why condition (\ref{itm:simply_conn}) of the preceding definition is necessary.

\begin{rem}\label{rem:inj_ordinary_htpy}
By Proposition \ref{prop:etale_kpi1}, Proposition \ref{prop:inj_sec_conj} and Lemma \ref{lemma:section_centralizer} imply that smooth hyperbolic curves and smooth proper genus $1$ curves satisfy the injectivity in the homotopy section conjecture.
\end{rem}

\begin{rem}\label{rem:dim0}
By \cite[Proposition 6.1.19]{corwinthesis}, a geometrically reduced finite scheme over a field satisfies the injectivity and surjectivity in the homotopy section conjectures. This shows how the homotopical version of the section conjectures deals nicely with geometric connected components.
\end{rem}

\begin{defn}[Analogue of Remark \ref{rem:compat_cusp_curve}]\label{hdefn:compatibledatum}
If $(X,C,k)$ is a homotopy cuspidal datum and $k'/k$ an extension of fields of characteristic $0$, we say that a homotopy cuspidal datum $(X_{k'},C',k')$ is compatible with $(X,C,k)$ if the base change map of Proposition \ref{prop:hsecbasechange} sends $C$ into $C'$.
\end{defn}

\subsection{The Homotopy Section Conjecture in Fibrations}

We will show that the homotopy section conjecture persists in fibration sequences. First, we need to specify how cuspidal data is supposed to behave in fibrations.

\begin{defn}\label{hdefn:cusp_fib}
Let $k$ be a field and $S$ a $k$-scheme. Let $f \colon X \to S$ be a geometric fibration. Suppose we are given cuspidal data $(S,C_S,k)$ for $S$ and $(X_s,C_{X_s},k)$ for the fiber $X_s$ above every $s \in S(k)$. Then we define the \emph{homotopy cuspidal datum $(X,C_{X,f},k)$ induced by $f$} as follows. 

For $s \in S(k)$, let $X_s$ denote the fiber of $f$ above $s$ and $\iota_s \colon X_s \to X$ the inclusion. The map $f$ induces a map $f:\hseck{X} \to \hseck{S}$, and the map $\iota_s$ induces a map $\iota_s \colon \hseck{X_s} \to \hseck{X}$. We define $C_{X,f}$ by
$$C_{X,f} = f^{-1}(C_S) \bigcup \left(\bigcup_{s \in S(k)} \iota_s(C_{X_s})\right) \subseteq \hseck{X}.$$

Note that Remark \ref{rem:cusp_data_curves} always applies. If $C_S$ and $C_{X_s}$ are all empty, such as for smooth proper curves, then $C_{X,f}$ is empty.
\end{defn}

\begin{rem}\label{rem:compat_cusp_fib}
Suppose $k'/k$ is an extension of characteristic $0$ fields. Suppose that we begin with compatible cuspidal data $(S,C_S,k)$, $(S,C_S',k')$, $(X_s,C_{X_s},k)$, and $(X_s,C_{X_s}',k')$. Then by a simple diagram chase, $(X,C_{X,f},k)$ is compatible with $(X,C_{X,f}',k')$.
\end{rem}

We now prove the main result of this subsection.

\begin{thm}\label{hthm:fib_inj}
Suppose we have $f \colon X \to S$, $(S,C_S,k)$, and $(X_s,C_{X_s},k)$ for $s \in S(k)$ as in Definition \ref{hdefn:cusp_fib}. Suppose that $f$ satisfies the hypotheses of Lemma \ref{hlemma:les}. If $(S,C_S,k)$ and $(X_s,C_{X_s},k)$ for every $s \in S(k)$ satisfy the injectivity (resp. surjectivity) in the homotopy section conjecture, then so does $(X,C_{X,f},k)$.
\end{thm}

\begin{proof}
All of these are a simple diagram chase using Lemma \ref{hlemma:les}. Note that in proving\ref{itm:kummer_inj} and \ref{itm:cusp_faithful} of Definition \ref{defn:inj_hsec_conj}, one needs to use \ref{itm:simply_conn} for $S$ to show that $\iota_s \colon X_s(hk) \to X(hk)$ is injective for $s \in S(k)$. More details are given in the proofs of Theorems 6.2.6 and 6.2.7 of \cite{corwinthesis}.
\end{proof}

\subsection{Elementary Fibrations}

This section gives us a tool for finding open subsets of arbitrary smooth geometrically connected varieties that satisfy the (homotopy) section conjecture.

The following definition is similar to one from \cite{sga4-3}.

\begin{defn}[Definition 11.4 of \cite{fr82}]\label{defn:elm_fib}
An \emph{(hyperbolic) elementary fibration} is a geometric fibration whose fibers are geometrically connected (hyperbolic) affine curves.
\end{defn}


\begin{defn}\label{defn:goodnbhd}
A \emph{(hyperbolic) good neighborhood} over $k$ is the composition of a sequence of (hyperbolic) elementary fibrations ending in $\Spec{k}$. Following \cite{Hoshi14}, we refer to a sequence of elementary fibrations whose composition is the good neighborhood $f$ as a \emph{sequence of parametrizing morphisms} for $f$.
\end{defn}

\begin{rem}\label{rem:proj_prop}
Assuming $f$ is a special geometric fibration, $Y$ defines a relatively ample divisor of $\cl{X}$, which means that $\cl{f}$ is automatically projective.
\end{rem}



\begin{defn}\label{defn:cusp_framing}
Let $f \colon X \to \Spec{k}$ be a good neighborhood with a sequence of parametrizing morphisms. Then we define the \emph{homotopy cuspidal datum $(X,C_{X,f},k)$ induced by $f$ and the sequence of parametrizing morphisms} to be the homotopy cuspidal datum induced successively by Definition \ref{hdefn:cusp_fib}, starting with $(\Spec{k},\emptyset,k)$.
\end{defn}

\begin{rem}
By an abuse of notation, we do not include the sequence of parametrizing morphisms in the notation for $C_{X,f}$. However, we believe this should not lead to too much confusion, and Lemma \ref{hcor:hgn_inj_sec_conj} is true regardless of the sequence of parametrizing morphisms. \cite[Conjecture 6.4.1]{corwinthesis} states that $C_{X,f}$ does not depend on the sequence of parametrizing morphisms, and the forthcoming paper \cite{HigherCusp} will discuss this and other questions about cuspidal sections for higher dimensional varieties in more detail.
\end{rem}

\begin{lemma}\label{hcor:hgn_inj_sec_conj}
Let $X$ be a hyperbolic good neighborhood over a subfield $k$ of a finite extension of $\Qp$. Then $X$ is an \'etale $K(\pi,1)$, and $(X,C_{X,f},k)$ satisfies the injectivity in the homotopy section conjecture. Furthermore, if Conjecture \ref{c:gro2} holds for every curve over $k$, then $(X,C_{X,f},k)$ satisfies the surjectivity in the homotopy section conjecture.
\end{lemma}

\begin{proof}

We induct on the dimension of $X$. If $X$ is dimension $0$, then it is $\Spec{k}$, so all of the results hold automatically.

Now suppose that $X$ has dimension $n \ge 1$. Then there is a hyperbolic elementary fibration $X \xrightarrow{f} Y$, where $Y$ satisfies the same hypotheses and has dimension $n-1$. We may assume by induction that $Y$ satisfies the conclusion of the lemma.

By \cite[Lemma 2.7(a)]{schsx16}, the fibers of $f$ are \'etale $K(\pi,1)$'s. Then \cite[Theorem 11.5]{fr82} shows that $X$ is an \'etale $K(\pi,1)$. 


By Remark \ref{rem:inj_ordinary_htpy}, the fibers of $f$ satisfy the injectivity in the section conjecture. Therefore, by Theorem \ref{hthm:fib_inj} and our induction hypothesis, $X$ does as well.

Similarly, if Conjecture \ref{c:gro2} holds, the fibers of $f$ satisfy the surjectivity in the homotopy section conjecture by Remark \ref{rem:surj_ordinary_htpy}. Again, by Theorem \ref{hthm:fib_inj} and our induction hypothesis (under the assumption of this conjecture), $X$ does as well.
\end{proof}

\begin{lemma}\label{lemma:hypgoodnbdhd}
Every smooth geometrically connected variety $X$ over an infinite field $k$ has an open cover by hyperbolic good neighbourhoods over $k$.
\end{lemma}

\begin{proof}
This is \cite[Lemma 6.3]{schsx16} or \cite[Lemma 6.3.15]{corwinthesis}. This result was already referred to as a ``Remark of M.Artin'' in \cite{Gro83}.
\end{proof}

\begin{rem}\label{rem:compat_cusp}
By Remarks \ref{rem:compat_cusp_curve} and \ref{rem:compat_cusp_fib}, the cuspidal data of Definition \ref{defn:cusp_framing} are compatible for varying $k$ of characteristic $0$.
\end{rem}

We can now answer Question \ref{quest:vsa} for $(\fcov,S,k)$ while circumventing Theorem \ref{thm:vsacover}:

\begin{cor}\label{cor:fcovallk2}
Let $X$ be a smooth geometrically connected variety over a number field $k$, and assume Conjecture \ref{c:gro2}. Then $X$ has a Zariski open cover $X = \bigcup_i U_i$ such that $U_i$ is VSA for $(\fcov,S,k)$ for any nonempty set $S$ of finite places of $k$.
\end{cor}

\begin{proof}
By Lemma \ref{lemma:hypgoodnbdhd}, $X$ has an open cover by hyperbolic good neighborhoods $U_i$. By Lemma \ref{hcor:hgn_inj_sec_conj}, $U_i$ satisfies the surjectivity in the section conjecture over $k$ and the injectivity in the section conjecture over $k_v$ for all finite places $v$ of $k$ (and compatible cuspidal data by Remark \ref{rem:compat_cusp}). Since $U_i$ is an \'etale $K(\pi,1)$ by Lemma \ref{hcor:hgn_inj_sec_conj}, Proposition \ref{prop:etale_kpi1} tells us that homotopy sections and fundamental group sections are the same for $U_i$. We may therefore apply the same argument as in Proposition \ref{prop:sec_conj_vsa} to show that $U_i$ satisfies VSA for $(\fcov,S,k)$.
\end{proof}





\part{Examples}\label{part:examples}

\section{Poonen's Counterexample}\label{sec:poonen_ce}

In \cite{poo08}, Poonen found the first example of a variety with no rational points that does not satisfy VSA for the \'etale-Brauer obstruction. In this section, we review the construction of this variety. We also use Theorem \ref{thm:vsa_geom_fib} to show, under certain conditions, that any example like Poonen's needs to have at least one singular fiber.

\subsection{Conic Bundles}
We now present some general notions about conic bundles, as described in \cite[\S 4]{poo08}. We base our notation on \cite{poo08} and then add some notation of our own.

Let $k$ be a field. From now on, let $\eps$ equal $1$ if $k$ has characteristic $2$ and $0$ otherwise.

Let $B$ be a nice $k$-variety. Let $\mathcal{L}$ be
a line bundle on $B$. Let $\mathcal{E}$ be the rank $3$ locally free sheaf
$$ \mathcal{O} \oplus \mathcal{O} \oplus \mathcal{L} $$
on $B$. Let $a \in  k^\times$, and let $s \in  \Gamma(B,\mathcal{L}^{\otimes 2})$ be a nonzero global section. Consider the section
$$ \eps \oplus 1 \oplus (-a) \oplus (-s) \in \Gamma(B,\mathcal{O} \oplus \mathcal{O} \oplus \mathcal{O} \oplus \mathcal{L}^{\otimes 2}) \subset
\Gamma (B, \Sym^2 \mathcal{E}),$$
where the first $\mathcal{O}$ corresponds to the product of the first two summands of $\mathcal{E}$, and the last three terms $\mathcal{O} \oplus \mathcal{O} \oplus \mathcal{L}^{\otimes 2}\subset \Sym^2 \mathcal{E}$ are the symmetric squares of the  three individual summands of $\mathcal{E}$. The zero locus of  $\eps \oplus 1 \oplus (-a) \oplus (-s)$ in $\PP \mathcal{E}^v$ is a projective geometrically integral $k$-variety $X=X(B,\mathcal{L},a,s)$ with a morphism $\alpha : X \to  B$.

\begin{defn}\label{defn:conic_bundle_datum}
We call an element
$$ (\mathcal{L},s,a)\in \Div{B} \times \Gamma(B,\mathcal{L}^{\otimes 2}) \times k^{\times},$$
where $s \neq 0$, a \emph{conic bundle datum on $B$} and $X$ \emph{the total space of $(\mathcal{L},s,a)$}. We denote $X=Tot_B(\mathcal{L},s,a)$.

\end{defn}

If $U$ is a dense open subscheme of $B$ with a trivialization $\mathcal{L}|_U \cong \mathcal{O}_U$, and we identify $s|_U$ with an element of $\Gamma(U,\mathcal{O}_U)$, then the affine scheme defined by $y^2+\eps yz - az^2 = s|_U$ in $\Ab^2_U$ is a dense open subscheme of $X$. We therefore refer to $X$ as the conic bundle given by $y^2+\eps yz  - az^2 = s$.

In the special case where $B = \PP^1$, $\mathcal{L} = \mathcal{O}(2)$, and the homogeneous form $s \in \Gamma(\PP^1,\mathcal{O}(4))$ is
separable, $X$ is called the Ch\^{a}telet surface given by $y^2+\eps yz -az^2 = s(x)$, where $s(x) \in  k[x]$ denotes
a dehomogenization of $s$.

Returning to the general case, we let $Z$ be the subscheme $s = 0$ of $B$. We call $Z$ the degeneracy
locus of the conic bundle $(\mathcal{L},s,a)$. Each fiber of $\alpha$ above a point of $B-Z$ is a smooth plane conic, and each fiber above a geometric point of $Z$ is a union of two projective lines crossing transversally
at a point. A local calculation shows that if $Z$ is smooth over $k$, then $X$ is smooth over $k$.

\subsection{Poonen's Variety}\label{sec:the_variety}

Let $k$ be a global field, let $a \in k^\times$, and let $\tilde{P}_\infty(x), \tilde{P}_0(x) \in k[x]$ be  relatively prime separable degree $4$ polynomials such that the (nice) Ch\^{a}telet surface $\mathcal{V}_\infty$ given by $y^2+\eps yz  - az^2 = \tilde{P}_\infty(x)$ over $k$ satisfies $\mathcal{V}_\infty(\Ab_k) \neq \emptyset$  but $\mathcal{V}_\infty(k) = \emptyset$. Such Ch\^{a}telet surfaces always exist : see \cite[Propositions 5.1 and 11]{poo08} if the characteristic of $k$ is not $2$ and \cite{viray12} otherwise. If $k = \QQ$, one may use the original example from \cite{Isk71} with $a = -1$ and $\tilde{P}_\infty(x)  \colonequals  (x^2 - 2)(3 - x^2)$.

Let $P_\infty(w, x)$ and $P_0(w, x)$  be the homogenizations of $\tilde{P}_\infty$ and $\tilde{P}_0$. Let $\mathcal{L} = \Oc(1, 2)$ on $\PP^1 \times \PP^1$ and define
$$s_1  \colonequals  u^2P_\infty(w, x) + v^2P_0(w, x) \in \Gamma(\PP^1 \times \PP^1,\mathcal{L}^{\otimes2}),$$ where the two copies of $\PP^1$ have homogeneous coordinates $(u: v)$ and $(w: x)$, respectively. Let $Z_1 \subset \PP^1 \times \PP^1$ be the zero locus of $s_1$. Let $F \subset \PP^1$ be the (finite) branch locus of the first projection $Z_1 \to \PP^1$. i.e., $$ F  \colonequals  \left\{(u:v) \in \PP^1 | u^2P_\infty(w, x) + v^2P_0(w, x) \text{  has a multiple root} \right\}.$$ Let $\alpha_1 : \mathcal{V} \to \PP^1 \times \PP^1$ be the conic bundle given by $y^2+\eps yz  - az^2 = s_1$, a.k.a. the conic bundle on $\PP^1 \times \PP^1$ defined by the datum $(\Oc(1, 2),a,s_1)$.

Composing $\alpha_1$ with the first projection $\PP^1 \times \PP^1 \to \PP^1$ yields a morphism
$\beta_1 : \mathcal{V} \to \PP^1$ whose fiber above $\infty  \colonequals  (1 : 0)$ is the Ch\^{a}telet surface $\mathcal{V}_\infty$ defined earlier.

Now let $C$ be a nice curve over $k$ such that $C(k)$ is finite and nonempty. Choose a dominant
morphism
$ \gamma \colon C \to \PP^1$, \'{e}tale above $F$, such that
$\gamma(C(k)) = \{\infty\}$.
Define  $X  \colonequals  \mathcal{V} \times_{\PP^1} C$ to be the fiber product with respect to the maps $\beta_1 \colon \mathcal{V}\to \PP^1, \gamma \colon C \to \PP^1$,
and consider the morphisms $\alpha$ and $\beta$ as in the diagram:

$$\xymatrix{
X   \ar@/_2pc/[dd]_{\beta} \ar[d]_\alpha  \ar[r]  &   \mathcal{V} \ar[d]_{\alpha_1} \ar@/^2pc/[dd]^{\beta_1}   \\
C \times \PP^1 \ar[d]_{\pr_1} \ar[r]^{(\gamma,1)}    &   \PP^1\times \PP^1 \ar[d]_{\pr_1}   \\
C        \ar[r]^\gamma          &   \PP^1 \\
}$$

The variety $X$ is the one constructed in \cite[\S 6]{poo08}. The same paper proves that $X(\Ab_k)^{\etbr}  \neq \emptyset$ (Theorem 8.2) and  $X(k) = \emptyset$ (Theorem 7.2). We present the proof that $X(k) = \emptyset$ because it is short and simple.

\begin{prop}\label{prop:no_points}
If $X$ is the variety constructed above, then $X(k) = \emptyset$.
\end{prop}
\begin{proof}
Assume $x_0\in X(k)$; we have $c_0  \colonequals  \beta(x_0) \in C(k)$, but then $x \in \beta^{-1}(c_0)$. By the construction of $X$, $\beta^{-1}(c_0)$ is isomorphic to $\beta_1^{-1}(\gamma(c_0)) = \beta_1^{-1}(\infty) \cong \mathcal{V}_\infty$, but $\mathcal{V}_\infty(k)=\emptyset$ by construction.
\end{proof}

Note that $X$ can also be considered as the variety corresponding to the datum $(O(1, 2),a,s_1)$ pulled back via $(\gamma,1)$ to $C\times \PP^1$.

\subsection{The Necessity of Singular Fibers}\label{sec:singular_fiber_needed}

One of the original motivations for the \'etale homotopy obstruction of \cite{hs13} was the fact that Poonen's counterexample was not a fibration (as it contained singular fibers). We complete this circle of reasoning to show that, assuming the Tate-Shafarevich conjecture and a technical condition on the real points, these singular fibers were necessary. Our main tool is Theorem \ref{thm:vsa_geom_fib}, which relies on the \'etale homotopy obstruction.

\begin{thm}\label{thm:singular_fiber_needed}
Let $f\colon X \to C$ be a smooth proper family of Ch\^{a}telet surfaces over an elliptic curve $C$ with $|C(k)|<\infty$. Suppose that for all $a \in C(k)$, we have $X_a(k)= \emptyset$. Suppose furthermore that the Tate-Shafarevich group of $C$ has trivial divisible subgroup, and that for every real place $v$ of $k$, every $a \in S(k)$, and every $b \in X_a(k_v)$, the map $\pi_1(X(k_v),b) \to \pi_1(S(k_v),a)$ is surjective (vacuous if $k$ is totally imaginary). Then $X(\Ab_k)^{\etbr} = \emptyset$.
\end{thm}

\begin{proof}
We must first verify that the hypotheses of Theorem \ref{thm:vsa_geom_fib} hold for $f$. We refer to them by number.

By \cite[Corollary 8.1]{sto07}, we have $$\quad C(k) = C(\Ab_k)^{\br}_\bullet.$$ In particular, we find that $C(k) = C(\Ab_k^f)^{\br}$, so $C$ is VSA for $(h,f,k)$, and hence \ref{item:vsa_base} holds.

As $C(k)$ is finite, it equals the completed Selmer group of $C$ (c.f. \cite{sto07}, Section 2). But $C(hk)=H^1(k;H_1(C^s;\widehat{\Zb})) = \varprojlim_{n} H^1(k;C[n])$, and the completed Selmer group is just $\Sel(C/k)$. This together with Proposition \ref{prop:inj_sec_conj} implies that \ref{item:img_sel} holds.


\ref{item:finite_inj} and \ref{item:tot_imag} follow immediately because $T=f$ and because $S$ is a smooth geometrically connected curve, respectively.

Every geometric fiber is a Ch\^{a}telet surface, so $\pi_1(X_a((k_v)_s))$ is trivial. By that and the hypothesis on real fundamental groups, \ref{item:real_places} is verified.

Finally, every fiber $X_a$ for $a \in C(k)$ is a Ch\^{a}telet surface without rational points, so by \cite{CTSSD87a,CTSSD87b}, we have $X_a(\Ab_k)^{\br}=X_a(\Ab_k)^{\etbr} = \emptyset$. This shows that \ref{item:vsa_fiber} holds.

By Theorem \ref{thm:vsa_geom_fib}, $X$ satisfies VSA for $(\mathrm{\acute{e}t-Br},f,k)$. By the argument of Proposition \ref{prop:no_points}, we have $X(k)=\emptyset$, hence $X(\Ab_k)^{\etbr}=\emptyset$.
\end{proof}

\section{VSA Stratifications and Open Covers}\label{sec:vsa_strat_open}

In this section, we seek to positively answer Question \ref{quest:vsa} for Poonen's variety in an \emph{explicit} fashion. That is, we want to prove that a \emph{specific} open cover or stratification satisfies VSA. In fact, we present a few different results in this direction, all of which differ slightly in their exact hypotheses, conclusions, and methods. More precisely, the different results differ in:
\begin{enumerate}
    \item\label{item:which_obstruction} Whether they prove VSA for the Brauer-Manin or \'etale-Brauer obstructions,
    \item\label{item:open_vs_stratification} Whether they prove VSA for open covers or stratifications,
    \item Whether they apply over all number fields or not, and
    \item\label{item:conjectures} Whether they assume conjectures such as the Tate-Shafarevich conjecture or the section conjecture (Conjecture \ref{c:gro2}).

\end{enumerate}

In both \ref{item:which_obstruction} and \ref{item:open_vs_stratification}, the first condition is stronger than the second conditition. In \ref{item:conjectures}, for the sake of explicitness, assuming the Tate-Shafarevich conjecture is actually much more harmless, as there are specific elliptic curves for which the Tate-Shafarevich conjecture has been proven.

More specifically, in Proposition \ref{prop:vsa_strat_poonen}, we find a stratification that satisfies VSA for Brauer-Manin assuming Tate-Shafarevich. In Corollary \ref{cor:vsa_example}, we find an open cover satisfying VSA for \'etale-Brauer assuming Conjecture \ref{c:gro2} and crucially using the \'etale homotopy obstruction. Finally, in Theorem \ref{thm:quasi_torsors_poonen}, we find a stratification that satisfies VSA for \'etale Brauer-Manin over totally imaginary number fields $k$ assuming Tate-Shafarevich. While Theorem \ref{thm:quasi_torsors_poonen} is technically weaker than Proposition \ref{prop:vsa_strat_poonen} with respect to our criteria listed above, it introduces the concept of the ramified \'etale-Brauer obstruction, which we believe is useful in its own right. See also Remarks \ref{rem:variant_2} and \ref{rem:variant_5} for other variations on the hypotheses and conclusions.

\subsection{VSA Stratifications and Open Covers}\label{sec:a_vsa_open_cover}

Let $U$ be an open subset of $C\setminus C(k)$, and let $S$ be an open subset of the smooth locus containing $C(k)$ (we always assume that $C(k)$ is finite).

\begin{lemma}\label{lemma:vsa_tateshaf}
Suppose $k$ is a number field. Let $J$ be the Jacobian of $C$, and suppose that $J(k)$ is finite and that the Tate-Shafarevich group $\Sha(J)$ has trivial divisible subgroup. Then $U$ has empty Brauer set.
\end{lemma}

\begin{proof}
By \cite[Corollary 8.1]{sto07}, we have $$\quad C(k) = C(\Ab_k)^{\br}_\bullet.$$ In particular, we find that $C(k) = C(\Ab_k^f)^{\br}$, and therefore that $U(\Ab_k^f)^{\br}=U(\Ab_k)^{\br}=U(k)=\emptyset$.
\end{proof}

\begin{lemma}\label{lemma:vsa_dim_0}
With $S$ as above and $Z \colonequals C \setminus S$, we have $Z(\Ab_k)^{\fab} = \emptyset$.
\end{lemma}

\begin{proof}
This follows immediately from \cite[Theorem 8.2]{sto07}.
\end{proof}

\begin{prop}\label{prop:vsa_strat_poonen}
Under the conditions of Lemma \ref{lemma:vsa_tateshaf}, there is a stratification $X = X_1 \cup X_2$ such that $X_1(\Ab_k)^{\br} = \emptyset$ and $X_2(\Ab_k)^{\br} = \emptyset$.
\end{prop}

\begin{proof}Let $X_1$ be $\beta^{-1}(U)$, and let $X_2$ be $\beta^{-1}(C(k))$. By Lemma \ref{lemma:vsa_tateshaf}, we know that $X_1(\Ab_k)^{\br} = \emptyset$.

We also know by construction that $X_2$ is a union of copies of $\Vc_\infty$. As $\Vc_\infty(\Ab_k)^{\br}=\emptyset$, we find that $X_2(\Ab_k)^{\br}=\emptyset$.
\end{proof}

\begin{rem}\label{rem:variant_2}
One may get a similar result as Proposition \ref{prop:vsa_strat_poonen} by assuming Conjecture \ref{c:gro2} instead of $\Sha(J)_{\mathrm{div}} = 0$, either by applying Conjecture \ref{c:gro2} directly to $U$, or by applying Theorem \ref{thm:vsa_example} to $S$ and Lemma \ref{lemma:vsa_dim_0} to $C \setminus S$.
\end{rem}

Next, we present an explicit proof that $X$, as constructed in Section \ref{sec:the_variety}, has a Zariski open cover with empty \'etale-Brauer set using Theorem \ref{thm:vsa_geom_fib} and assuming Conjecture \ref{c:gro2}. Section \ref{sec:ramified_covers} will then prove a stronger result, namely that there is a finite ramified cover with empty \'etale-Brauer set, without assuming Conjecture \ref{c:gro2}.

\begin{thm}\label{thm:vsa_example}
Suppose that every proper nonempty open subvariety of $C$ satisfies Conjecture \ref{c:gro2}. Then there is an open subvariety $S \subseteq C$ such that $S \cup U = C$, and $X_S \colonequals S \times_C X$ satisfies VSA for $(\mathrm{\acute{e}t-Br},f,k)$.
\end{thm}

\begin{proof}

With the notation of Section \ref{sec:the_variety}, let $F' \colonequals \gamma^{-1}(F)$, and let $C' \colonequals C \setminus F'$. Then $\beta$ is smooth over $C'$.

Let $v$ be a real place of $k$. As $C$ is smooth, $C'(k_v)$ is a disjoint union of components each homeomorphic to a line or circle. For each circle component containing a rational point, we wish to remove a closed non-rational point of $C'$ that meets that component.

To do this, we need to show that each component contains a $k_s$-point that is not a $k$-point. Let $k_v^{\mathrm{alg}}$ denote the algebraic closure of $k$ in $k_v$. By the completeness of the theory of real closed fields, any first-order formula true of $k_v$ is also true of $k_v^{\mathrm{alg}}$. By Theorems 2.4.4 and 2.4.5 of \cite{bcr98}, the property that a point lies in a given connected component is a semi-algebraic condition. As well, the property that a point is not equal to any element of $C'(k)$ is also a semi-algebraic condition. As each component has uncountably many $k_v$-points and hence at least one point not in $C'(k)$, each component has a $k_v^{\mathrm{alg}}$-point not in $C(k)$. We choose one for each component with a rational point and remove the corresponding closed points of $C'$.

We let $S$ denote the resulting open subvariety of $C$. It is clear that $S(k)=C(k)$, so that $S \cup U = C$. It suffices to verify the hypotheses of Theorem \ref{thm:vsa_geom_fib} for the projection $f = \beta \mid_{X_S} \colon X_S \to S$ and $T=f$.

By Lemma \ref{lemma:combined2}(\ref{item:embeds_vsa}), $S$ is VSA for $(h,f,k)$, so \ref{item:vsa_base} holds. By assumption, \ref{item:sec_conj} holds. \ref{item:finite_inj} and \ref{item:tot_imag} follow immediately because $T=f$ and because $S$ is a smooth geometrically connected curve, respectively.


By construction, every connected component of $S(k_v)$ is simply connected for each real place $v$. As well, every geometric fiber is a Ch\^{a}telet surface, so $\pi_1(X_a((k_v)_s))$ is trivial, and hence \ref{item:real_places} is verified.

Finally, every fiber over a rational point of $S$ is isomorphic to $\Vc_\infty$. But we assumed that $\Vc_{\infty}(k)=\emptyset$, so by \cite{CTSSD87a,CTSSD87b}, we have $\Vc_{\infty}(\Ab_k)^{\br}=\Vc_{\infty}(\Ab_k)^{\etbr} = \emptyset$. This shows that \ref{item:vsa_fiber} holds.
\end{proof}

\begin{cor}\label{cor:vsa_example}
Under the hypotheses and notations of Theorem \ref{thm:vsa_example}, the variety $X$ has an open cover $X = X_S \cup X_U$ such that $X_S$ and $X_U \colonequals U \times_C X$ both have empty \'etale-Brauer set.
\end{cor}

\begin{proof}
We have already shown that $X_S(k) \subseteq X(k) = \emptyset$, so by Theorem \ref{thm:vsa_example}, we have $X_S(\Ab_k)^{\etbr} = \emptyset$. By Conjecture \ref{c:gro2} and Proposition \ref{prop:sec_conj_vsa} applied to $U$, we find that $U(\Ab_k)^{\fcov}=\emptyset$, which implies that $X_U(\Ab_k)^{\fcov}=X_U(\Ab_k)^{\etbr}=\emptyset$.
\end{proof}

\begin{rem}\label{rem:tate_shaf}
For the proof of Corollary \ref{cor:vsa_example} (but not Theorem \ref{thm:vsa_example}), one may avoid using Conjecture \ref{c:gro2} for $U$ by assuming the hypotheses of Lemma \ref{lemma:vsa_tateshaf}.  However, the argument in Theorem \ref{thm:vsa_example} already uses Conjecture \ref{c:gro2}, so it makes sense to use that conjecture again.

\end{rem}


\section{The Brauer-Manin obstruction applied to ramified covers}\label{sec:ramified_covers}

In this section, we introduce the notion of the ramified \'etale-Brauer obstruction. We also formulate a conjecture analogous to Theorems \ref{thm:fabsomek} and \ref{thm:fcovallk}. In Section \ref{sec:quasi_torsors_poonen}, we then explain how to apply this to Poonen's example.

\subsection{Quasi-Torsors and the Ramified Etale-Brauer Obstruction}\label{sec:quasi_torsors}

We shall now define slight generalizations of the concepts of torsors and the \'{e}tale Brauer-Manin obstruction.

\begin{defn}\label{defn:quasi_torsor}
Let $X$ be a variety over a field $k$, $G$ a finite smooth $k$-group, and $D \subseteq X$ an a closed subvariety. A \emph{$G$-quasi-torsor over $X$ unramified outside $D$} is a map $\pi:Y \to X$ and a $G$-action on $Y$ respecting $\pi$ such that
\begin{enumerate}
\item $\pi$ is a surjective and quasi-finite morphism.
\item The pullback of $\pi$ from $X$ to $X \setminus D$ is a $G$-torsor over $X\setminus D$.
\end{enumerate}
We call $d = |G|$ the degree of $Y$.
\end{defn}

Given $\rho \colon Z \to X$ an arbitrary morphism of $k$-varieties, the pullback $\rho^{-1}(Y)$ is a $G$-quasi-torsor unramified outside $\rho^{-1}(D)$.

Let $\pi:Y \to X$ be a $G$-quasi-torsor over $X$ unramified outside $D$. As in the case of a usual $G$-torsor, one can twist $\pi:Y \to X$ by any $\sig \in H^1(k,G)$ and get a $G^{\sig}$-quasi-torsor $\pi^\sig: Y^\sig \to X$, also unramified outside $D$.

If we assume that  $D(k)= \emptyset$, then as in Section \ref{sec:altdescr}, we get:
$$X(k) = \bigcup \limits_{\sigma \in H^1(k,G)} \pi^{\sigma}(Y^{\sigma}(k)) $$

If $k$ is a global field, it follows that:

$$X(k) \subset  X(\Ab_k)^{\pi,\Br}   \colonequals  \bigcup \limits_{\sigma \in H^1(k,G)} \pi^{\sigma}(Y^{\sigma}(\Ab_k)^{\Br}).$$

\begin{defn}\label{defn:ram_et_br}
Letting $\pi$ range over all quasi-torsors over $X$ unramified outside $D$, we define the \emph{($D$-)ramified \'etale-Brauer obstruction} by

$$ X(\Ab_k)^{\etbr \thicksim D} \colonequals \bigcap \limits_{\pi} X(\Ab_k)^{\pi,\Br} \subset X(\Ab_k)$$
\end{defn}

It follows from Section \ref{sec:altdescr} that $(X\setminus D)(\Ab_k)^{\etbr} \subseteq X(\Ab_k)^{\etbr \thicksim D}$.

Until now, we have left the fact that $D(k) = \emptyset$ as a black box. The idea behind this construction is that $D$ has smaller dimension, and it therefore might be easier to show that $D$ has empty \'etale-Brauer set. More generally, one might, e.g., find a subvariety $E$ of $D$, prove that $D(\Ab_k)^{\etbr \thicksim E} = \emptyset$, and finally prove that $E(\Ab_k)^{\etbr}=\emptyset$. We formalize this as follows.

\begin{defn}
We say that \emph{the absence of rational points is explained by the ramified \'etale-Brauer obstruction} if there exists a stratification $X = \coprod_i X_i$ and for each $i$, we have $X_i(\Ab_k)^{\etbr \thicksim D_i} = \emptyset$, where $D_i \colonequals \cl{X_i} \setminus X_i$.
\end{defn}

This is similar to the method of considering the \'etale-Brauer sets of dense open subsets, for $X(\Ab_k) ^{\etbr \thicksim D} = \emptyset$ implies that $(X\setminus D)(\Ab_k)^{\etbr} = \emptyset$. In particular, if the absence of rational points is explained by the ramified \'etale-Brauer obstruction, then the answer to Question \ref{quest:emptiness} is yes. However, one crucial difference is that the ramified \'etale-Brauer obstruction requires proving only emptiness of Brauer sets of \emph{proper} varieties. 

Inspired by Theorem \ref{thm:fabsomek}, we conjecture that the absence of rational points is always explained by the ramified \'etale-Brauer obstruction and that one may furthermore take a single quasi-torsor for each stratum:

\begin{conj}\label{conj:ram_etbr}
For any number field $k$ and variety $X/k$ with $X(k) = \emptyset$, there exists a stratification $X = \coprod_i X_i$ and for each $i$, a quasi-torsor $Y_i$ over the closure $\cl{X_i}$ of $X_i$ such that $Y_i$ restricts to a torsor over $X_i$ and $Y_i^{\sigma}(\Ab_k)^{\Br} = \emptyset$ for all twists $Y_i^{\sigma}$ of $Y_i$.
\end{conj}

This conjecture still does not imply that there is an algorithm to determine whether $X(k)$ is empty, as one may need to consider infinitely many twists of a quasi-torsor $Y_i$. Nonetheless, we now show that it is computable in the example of \cite{poo08} (in particular, see Remark \ref{rem:computable} to see why we need only finitely many twists). Thus, we know of no counterexamples to Conjecture \ref{conj:ram_etbr}.

\subsection{Quasi-torsors in Poonen's Example}\label{sec:quasi_torsors_poonen}

In this subsection, we show (under some conditions) that for the variety $X$ defined in Section \ref{sec:the_variety}, one can choose a divisor $D \subseteq X$ such that $D(\Ab_k)^{\br}=\emptyset$, and
$X(\Ab_k)^{\etbr \thicksim D} = \emptyset$.

More specifically, we assume as in Lemma \ref{lemma:vsa_tateshaf} that the Jacobian $J$ of $C$ satisfies $|J(k)|<\infty$ and $\Sha(J)_{\mathrm{div}} = 0$, which implies that $$\quad C(k) = C(\Ab_k)^{\br}_\bullet$$ and $$U(\Ab_k)^{\br}=\emptyset.$$

As before, let $F'  \colonequals  \gamma^{-1}(F) \subseteq C$ and $C'  \colonequals  C\backslash F'$. Note that $C'$ is a non-projective curve. Now let $D  \colonequals  \beta^{-1}(F')$. Note that $\infty \not \in F$, so that $ C(k) \cap F' = \emptyset$. By Lemma \ref{lemma:vsa_dim_0}, we have $F'(\Ab_k)^{\Br}=\emptyset$, hence $D(\Ab_k)^{\Br})=\emptyset$.


We will now spend the rest of Section \ref{sec:quasi_torsors_poonen} proving:

\begin{thm}\label{thm:quasi_torsors_poonen}
With notations as above, the absence of rational points on $X$ is explained by the ramified \'etale-Brauer obstruction if $k$ is a global field with no real places (i.e., a function field or a totally imaginary number field).
\end{thm}

\begin{rem}\label{rem:variant_5}
One may in fact combine Theorem \ref{thm:quasi_torsors_poonen} with Lemma \ref{lemma:vsa_tateshaf} to get a VSA open cover, rather than a stratification, without assuming Conjecture \ref{c:gro2} as in Corollary \ref{cor:vsa_example}.
\end{rem}

Now $X$ is a family over $C$ of conic bundles over $\PP^1$. The fiber over any element of $C(k)$ is isomorphic to the Ch\^{a}telet surface $\mathcal{V}_\infty$. All the fibers over $C'$ are smooth conic bundles.

Let $E' \subset (\PP^1 \backslash F) \times (\PP^1)^4 $ be the curve defined by
$$ u^2P_\infty(w_i, x_i) + v^2P_0(w_i, x_i) = 0 , 1 \leq i \leq 4 $$
$$ (w_i:x_i) \neq (w_j:x_j) , i \neq j, 1 \leq i,j \leq 4, $$
where $(u:v)$ are the projective coordinates of $\PP^1 \backslash F$ and $(w_i:x_i), 1 \leq i \leq 4 $ are the projective coordinates of the four copies of $\PP^1$. Since $\tilde{P}_\infty(x)$ and $\tilde{P}_0(x)$ are separable and coprime, we have that $E'$ is a smooth connected curve and that the first projection $E' \rightarrow \PP^1 \backslash F $ gives $E'$ the structure of an \'{e}tale Galois covering of $\PP^1 \backslash F$ with automorphism group $G = S_4$ that acts on the fibres by permuting $$ (w_i:x_i) , 1 \leq i \leq 4.$$
Since every birationality class of curves contains a unique projective smooth member, one can construct an $S_4$-quasi-torsor over $E \to \PP^1$ unramified outside $F$ which gives $E'$ when restricted to $\PP^1\backslash F$.

The $k$-twists of $E \to \PP^1 $ are classified by $H^1(k,S_4)$ which (since the action of $\Gamma_k$ on $S_4$ is trivial) coincides with the set $\Hom(\Gamma_k,S_4)/\sim$ of homomorphisms up to conjugation. More concretely, for every homomorphism $\phi : \Gamma_k \to S_4$, define $E_\phi$ to be the  $k$-form of $E$ with Galois action that restricts to the action
$$ \sig : ((u:v),((w_1:x_1),(w_2:x_2),(w_3:x_3),(w_4:x_4))) \mapsto $$ $$ ((u:v),((w_{\phi_\sigma(1)}:x_{\phi_\sigma(1)}),(w_{\phi_\sigma(2)}:x_{\phi_\sigma(2)}),
(w_{\phi_\sigma(3)}:x_{\phi_\sigma(3)}),(w_{\phi_\sigma(4)}:x_{\phi_\sigma(4)})))^{\sig} $$
on $E'$.

For every $\phi : \Gamma_k \to S_4$, we set $C_\phi  \colonequals  C\times_{\PP^1} E_\phi$ (relative to $\gamma: C \to \PP^1$ and the first projection $\pi_{\phi} \colon E_\phi \to  \PP^1$) and $X_\phi  \colonequals  X \times_{C} C_\phi$ (relative to $\beta: X \to C$ and the projection $C_\phi \to  C$).

Since the maps $\gamma : C \to \PP^1$ and $E \to \PP^1$ have disjoint ramification loci, all $C_\phi$ are geometrically integral, and so are all the $X_\phi$.

Then $X_\phi$ is a complete family of twists of a quasi-torsor over $X$ of degree $24$ unramified outside $D$. As we already know that $D(\Ab_k)^{\br}=\emptyset$, it suffices for the proof of Theorem \ref{thm:quasi_torsors_poonen} to show that $$ X_\phi(\Ab_k)^{\Br}   = \emptyset $$
for every $\phi \in H^1(\Gamma_k,S_4)$. We devote the rest of Section \ref{sec:quasi_torsors_poonen} to proving this fact.

\subsubsection{Reduction to \texorpdfstring{$X_{\phi_\infty}$}{X phi infinity}}

\begin{lemma}\label{l:curve_BM}
For every $\phi \in H^1(k,S_4)$, we have $C_\phi(k) = {C_\phi(\Ab_k)^{\Br} }_\bullet $.
\end{lemma}

\begin{proof}
By \cite[Corollary 7.3]{sto07}, we have $C_\phi(\Ab_k)^{\br}_{\bullet}=C_\phi(\Ab_k)^{\fab}_{\bullet}$ and $C(\Ab_k)^{\br}_{\bullet}=C(\Ab_k)^{\fab}_{\bullet}$. Then \cite[Proposition 8.5]{sto07} and the fact that $C(k)=C(\Ab_k)^{\br}_{\bullet}$ implies the result.
\end{proof}

Denote by $\phi_\infty \in H^1(k,S_4)$ the map $\Gamma_k \to S_4$ defined by the Galois action on the four roots of $P_\infty$.

\begin{lemma}\label{lemma:phi_not_phi_infty}
Let $\phi \in H^1(\Gamma_k, S_4)$ be such that $\phi \neq \phi_\infty$. Then $C_\phi(k) = \emptyset$.
\end{lemma}
\begin{proof}
Since $\phi \neq \phi_\infty$ we get that $E_\phi(k) \cap  \pi^{-1}_\phi(\infty) = \emptyset$. Since $\gamma(C(k)) = {\infty}$, we get that $C_\phi(k) = \emptyset$.
\end{proof}

Let $\rho_\phi \colon X_\phi \to C_\phi$ denote the map defined earlier. For every $\phi \in H^1(k,S_4)$, we have
$$\rho_\phi({X_\phi(\Ab_k)^{\Br} }_\bullet) \subseteq {C_\phi(\Ab_k)^{\Br} }_\bullet = C_\phi(k) = \emptyset,$$

so $X_\phi(\Ab_k)^{\Br}   =\emptyset$ for $\phi \neq \phi_{\infty}$.

\begin{rem}\label{rem:computable}
It is Lemma \ref{lemma:phi_not_phi_infty} along with VSA for $C$ coming from our condition on the Tate-Shafarevich group that lets us avoid all but one twist. In fact, this has a conceptual explanation. While infinitely many twists might have adelic points (as $\beta^{-1}(C')$ is not proper), all possible elements of the Brauer set lie in the fiber over $C(k)$, which is in fact proper. There are therefore finitely many twists that might have adelic points in this fiber.\end{rem}

\subsubsection{The proof that \texorpdfstring{$X_{\phi_\infty}(\Ab_k)^{\Br}   =\emptyset$}{the Brauer set is empty}}

In this subsection, we shall prove that if $k$ has no real places, then $X_{\phi_\infty}(\Ab_k)^{\Br}  =X_{\phi_\infty}(\Ab_k)_\bullet^{\Br}  = \emptyset$.

Let $p \in C_{\phi_\infty}(k)$. The fiber $\rho^{-1}_{\phi_\infty}(p)$ is isomorphic to the  Ch\^{a}telet surface $\mathcal{V}_\infty$. We shall denote by $\rho_p : \mathcal{V}_\infty \to X_{\phi_\infty}$ the corresponding natural isomorphism onto the fiber $\rho^{-1}_{\phi_\infty}(p)$. Recall that $\mathcal{V}_\infty$ satisfies $\mathcal{V}_\infty(\Ab_k)^{\Br}  = \emptyset$.

\begin{lemma}\label{l:on_fibre}
Let $k$ be a global field with no real places. Let $x \in X_{\phi_\infty}(\Ab_k)_\bullet^{\Br} $. Then there exists $p \in C_{\phi_\infty}(k)$ such that $x \in \rho_p(\mathcal{V}_\infty(\Ab_k)_\bullet )$.
\end{lemma}

\begin{proof}
From functoriality and Lemma \ref{l:curve_BM} we get
$$\rho_{\phi_\infty}(x) \in \rho_{\phi_\infty} (X_{\phi_\infty}(\Ab_k)_\bullet^{\Br} ) \subset  {C_{\phi_\infty}(\Ab_k)^{\Br} }_\bullet = C_{\phi_\infty}(k) $$
We denote $p = \rho_{\phi_\infty}(x) \in  C'_{\phi_\infty}(k)$. Now it is clear that in all but maybe the infinite places   $x \in \rho_p(\mathcal{V}_\infty(\Ab_k)  )$. Hence it remains to deal with the infinite places, which by assumption are all complex. But since both $X_{\phi_\infty}$ and $\mathcal{V}_\infty$ are geometrically integral, taking connected components reduces $X(\mathbb{C})$ and $\mathcal{V}_\infty(\mathbb{C})$ to a single point.
\end{proof}

\begin{lemma}\label{l:Br_surj}
Let  $p \in C_{\phi_\infty}(k)$ be a point. Then the map $$\rho_p^* : \br(X_{\phi_\infty}) \to \br(\mathcal{V}_\infty)$$
is surjective.

\end{lemma}

We will prove Lemma \ref{l:Br_surj} in Section \ref{sec:surj}.

\begin{lemma}
Let $k$ be global field with no real places. Then $X_{\phi_\infty}(\Ab_k)_\bullet^{\Br} = \emptyset$.
\end{lemma}

\begin{proof}
Assume that $X_{\phi_\infty}(\Ab_k)_\bullet^{\Br}  \neq \emptyset$. Let $x\in X_{\phi_\infty}(\Ab_k)_\bullet^{\Br} $. By Lemma \ref{l:on_fibre} there
exists a $p \in C_{\phi_\infty}(k)$ such that $x \in \rho_p(\mathcal{V}_\infty(\Ab_k)_\bullet)$. Let $y \in \mathcal{V}_\infty(\Ab_k)_\bullet$ be such that $\rho_p(y)= x$. We shall show that $y \in \mathcal{V}_\infty(\Ab_k)_\bullet^{\Br}$.

Indeed, let $b \in \br(\mathcal{V}_\infty)$. By Lemma \ref{l:Br_surj} there exists
a $\tilde{b} \in \br(X'_{\phi_\infty})$ such that $\rho_p^*(\tilde{b}) = b$. Now
$$ (y,b) = (y,\rho_p^*(\tilde{b})) = (\rho_p(y),\tilde{b}) = (x,\tilde{b}) = 0 $$
But by assumption $x\in X_{\phi_\infty}(\Ab_k)_\bullet^{\Br}$, so we have $(y,b) = (x,\tilde{b}) = 0$. Thus we have $y \in \mathcal{V}_\infty(\Ab_k)_\bullet^{\Br}  = \emptyset$ which is a contradiction.
\end{proof}

\subsubsection{The surjectivity of \texorpdfstring{$\rho_p^*$}{rho p *}}\label{sec:surj}
In this subsection, we shall prove the statement of Lemma \ref{l:Br_surj}. We switch gears for a moment and let $\alpha \colon X \to B$ be an arbitrary conic bundle given by datum $(\mathcal{L},s,a)$.

\begin{lemma} The generic fiber $X^s_\eta$ of $X^s \to  B^s$ is isomorphic to $\PP^1_{\kappa(B^s)}$, where $\kappa(B^s)$ is the field of rational functions on $B^s$ .
\end{lemma}

\begin{proof} It is a smooth plane conic and it has a rational point since $a$ is a square in $k_s \subset \kappa(B^s)$.
\end{proof}

\begin{lemma}
Denote the generic point of $B$ by $\eta$. Let $Z$ be the degeneracy locus. Assume that $Z^s$ is the union of the irreducible components $Z^s = \bigcup_{1\leq i\leq r} Z_i$. Then there is a natural exact sequence of Galois modules.

$$
\xymatrix{
0 \ar[r] & \bigoplus\ZZ Z_i  \ar[r]^-{\rho_1} & \Pic B^s \oplus \bigoplus \ZZ Z^{+}_i \oplus \bigoplus \ZZ Z^{-}_i \ar[r]^-{\rho_2} & \Pic X^s
 \ar[r]_-{\rho_3} \ar[rd]_{deg} &  \Pic {X^s}_\eta \ar[r] \ar@{=}[d] \ar@/_/[l]_-{\rho_4} & 0 \\
 & &  &  &  \mathbb{Z}&
}
$$

where $\rho_4$  is a natural section of $\rho_3$.

\end{lemma}

\begin{proof}
Call a divisor of $X^s$ vertical if it is supported on prime divisors lying above prime divisors
of $B^s$, and horizontal otherwise.
Denote by $Z^\pm_i$ the divisors that lie over $Z_i$ and defined by the additional condition that $y=\pm \sqrt{a} z$, respectively. Now define $\rho_1$ by
$$ \rho_1(Z_i) = (-Z_i, Z^{+}_i, Z^{-}_i) $$
and $\rho_2$ by
$$ \rho_2 (M, 0, 0) = \alpha^* M $$
$$ \rho_2 (0, Z^+_i, 0) = Z^+_i $$
$$ \rho_2 (0, 0,Z^-_i) = Z^-_i $$
Let $\rho_3$ be the map induced by $X^s_\eta \to X^s$. Each $\rho_i$ is $\Gamma_k$-equivariant. Given a prime divisor $D$ on $X^s_\eta$, we take $\rho_4(D) $ to be its Zariski closure in $X^s$. It is clear that $\rho_3 \circ \rho_4 = \Id$, so $\rho_3$ is indeed surjective.

The kernel of $\rho_3$ is generated by the classes of vertical prime divisors of $X$. In fact, there is exactly one
above each prime divisor of $B$, except that above each $Z_i \in  \Div B^s$ we have both $Z^+_i, Z^-_i \in  \Div X^s$. This proves exactness at $\Pic X^s$.

Now, since $\alpha:X^s\to B^s$ is proper, a rational function on $X^s$ with a vertical divisor must be the pullback of a rational function on $B^s$. Using the fact that the image of $\rho_2$ contains only vertical divisors, we prove exactness at
$$ \Pic B^s \oplus \bigoplus \ZZ Z^{+}_i \oplus \bigoplus \ZZ Z^{-}_i $$
The injectivity of $\rho_1$ is then trivial.
\end{proof}

We switch gears once again and let $X$ be as in Poonen's example.

\begin{lemma}
Let $p \in C_{\phi_\infty}(k)$ and $\rho_p : \mathcal{V}_\infty \to  X_{\phi_\infty}$ be the corresponding map as above. Then the map of Galois modules
$$\rho_p^* : \Pic(X_{\phi_\infty}^s) \to \Pic(\mathcal{V}_\infty^s)$$

has a section.
\end{lemma}

\begin{proof}
Consider the map $\phi_p: \PP^1 \to \PP^1\times C_{\phi_\infty}$ defined by $x \mapsto (x,p)$.
It is clear that the map $\rho_p:\mathcal{V}_\infty \to X_{\phi_\infty}$ comes from pulling back the conic bundle datum defining $X_{\phi_\infty}$ over $ \PP^1\times C_{\phi_\infty} $ by this map. Let $B =\PP^1\times C_{\phi_\infty}$, and consider the following commutative diagram with exact rows

$$\xymatrix{
0 \ar[r] & \bigoplus\ZZ Z_i  \ar[r] \ar[d] & \Pic B^s \oplus \bigoplus \ZZ Z^{+}_i \oplus \bigoplus \ZZ Z^{-}_i \ar[r] \ar[d] & \Pic X_{\phi_\infty}^s
 \ar[r]_-{deg} \ar[d]^{\rho_p^*}  & \mathbb{Z}  \ar[r] \ar@{=}[d]  \ar@/_/[l] & 0
\\
0 \ar[r] & \bigoplus\ZZ W_i \ar@/^/[u]^{s_1} \ar[r] & \Pic \PP^1_{k_s}  \oplus \bigoplus \ZZ W^{+}_i \oplus \bigoplus \ZZ W^{-}_i \ar@/^/[u]^{s_2} \ar[r] & \Pic \mathcal{V}_\infty^s \ar[r]_-{deg} &  \mathbb{Z} \ar@/_/[l]  \ar[r]& 0 \\
}$$
where  $Z$ is the degeneracy locus of $X_{\phi_\infty}$ over $B$, $W$ is the degeneracy locus of $\mathcal{V}_\infty$ over $\PP^1$, and $Z^s = \bigcup_{1\leq i\leq r} Z_i$ and $W^s = \bigcup_{1\leq i\leq r} W_i$ are decompositions into irreducible components. The existence of a section for $\rho_p^*$ follows by diagram chasing and the existence of the compatible sections $s_1$ and $s_2$.

Every $W_i$ ($1 \leq i \leq 4$ ) is  a point that corresponds to a different root  $(w_i:x_i)$ of the polynomial $P_\infty(x,w)$. We can choose  $Z_i \subset B^s$ to be  Zariski closure of the zero set of  $w_ix-x_iw$,  and similarly $Z^{\pm}_i \subset X_{\phi_\infty}^s$ to be Zariski closure of the zero set of  $y \pm\sqrt{a} z, w_ix-x_iw$.

Now we define: $Z_i = s_1(W_i)$ and $Z^{\pm}_i = s_2(W^{\pm}_i)$ and the map
$s_2:\Pic \PP^1_{k_s} \to  \Pic B^s$ is defined by the unique section of the map $\phi_p: \PP^1 \to \PP^1\times C_{\phi_\infty}$.

It is clear that $s_1$ and $s_2$ are indeed group-theoretic sections. To prove that $s_1$ and $s_2$ also respect the Galois action, we can write
$$ p = (c,((x^0_1:w^0_1),(x^0_2:w^0_3),(x^0_2:w^0_3),(x^0_2:w^0_3))) \in C(k)\times_{\mathbb{P}^1(k)}E_{\phi_\infty}(k),$$
and since $\gamma(C(k))=\{\infty\}$, the four points $\{(x^0_1:w^0_1),(x^0_2:w^0_3),(x^0_2:w^0_3),(x^0_2:w^0_3)\}$ are exactly the four different roots of $P_\infty(x,w)$ .
\end{proof}

\begin{lemma}[Lemma \ref{l:Br_surj}]
Let  $p \in C_{\phi_\infty}(k)$. Then the map $$\rho_p^* : \Br(X_{\phi_\infty}) \to \Br(\mathcal{V}_\infty)$$
is surjective.
\end{lemma}

\begin{proof}

Denote by $s_p: \Pic\overline{(\mathcal{V}_\infty)} \to \Pic(X_{\phi_\infty}^s)$ the section of

$$ \rho_p^* :\Pic(X_{\phi_\infty}^s) \to \Pic(\mathcal{V}_\infty^s) $$

It is clear that $s_p$ induces a section of the map

$$ \rho_p^{**} :H^1(k,\Pic(X_{\phi_\infty}^s)) \to H^1(k,\Pic(\mathcal{V}_\infty^s)) $$

By the Hochschild--Serre spectral sequence for $X$, we have:

$$H^1(k,\Pic(X^s)) = \ker[\Br X \to \Br X^s] / \im[\Br k \to \Br X]$$

Letting $$\Br_1(X) \colonequals  \ker[\Br X \to \Br X^s],$$

we get that the map $\rho_p^*:\br_1(X_{\phi_\infty}) \to \br_1(\mathcal{V}_\infty)$ is surjective. But since $\mathcal{V}_\infty^s$ is a rational surface (it is a Ch\^{a}telet surface), we have $\br\mathcal{V}_\infty^s = 0$, and thus  $\br_1(\mathcal{V}_\infty)=\br(\mathcal{V}_\infty)$. So we get that $\rho_p^* : \br(X_{\phi_\infty}) \to \br(\mathcal{V}_\infty)$ is surjective.
\end{proof}

{\def\section*#1{}

\bibliographystyle{alpha2}
\bibliography{references,etale_homotopy,chabauty-kim,ag}
}

\end{document}